\documentclass[reqno]{amsart}

\usepackage[top=30truemm,bottom=30truemm,left=35truemm,right=35truemm]{geometry}
\usepackage{amssymb}
\usepackage{amsmath}
\usepackage{mathrsfs}
\usepackage{amsfonts}
\usepackage[dvipdfmx]{graphicx}
\usepackage{xcolor}
\newcommand{\ep}{\varepsilon}
\newcommand{\dual}[1]{\langle{#1}\rangle}

\DeclareOldFontCommand{\rm}{\normalfont\rmfamily}{\mathrm}
\newcommand{\normx}[1]{\|{#1}\|_{\mathcal{X},\tau}}
\newcommand{\normxa}[1]{\|{#1}\|_{\mathcal{X},\tau,\star}}

\newcommand{\normy}[1]{\|{#1}\|_{\mathcal{Y},\tau}}

\newcommand{\normyb}[1]{\|{#1}\|_{\mathcal{Y},\tau,\#}}
\newcommand{\vnorm}[1]{|\hspace{-0.3mm}|\hspace{-0.3mm}|#1|\hspace{-0.3mm}|\hspace{-0.3mm}|}
\newcommand{\normxx}[1]{\vnorm{#1}_{\mathcal{X},\tau}}
\newcommand{\normxxa}[1]{\vnorm{#1}_{\mathcal{X},\tau,\star}}

\newcommand{\normyy}[1]{\vnorm{#1}_{\mathcal{Y},\tau}}

\newcommand{\normyyb}[1]{\vnorm{#1}_{\mathcal{Y},\tau,\#}}
\newcommand{\ns}[1]{#1}
\definecolor{cyan10}{cmyk}{.3,0,0,0}

\theoremstyle{plain}
\newtheorem{theorem}{Theorem}
\newtheorem{lemma}[theorem]{Lemma}
\theoremstyle{definition}
\newtheorem{remark}[theorem]{Remark}

\begin{document}

\title[Analysis of the DG time-stepping method]
{Variational analysis of the discontinuous Galerkin time-stepping method for
parabolic equations}
\author[N. Saito]{
Norikazu Saito}
\address{
Graduate School of Mathematical Sciences, The University of Tokyo\\
Komaba 3-8-1, Meguro-ku, Tokyo 153-8914, Japan}
\email{norikazu@g.ecc.u-tokyo.ac.jp}
\urladdr{http://www.infsup.jp/saito/index-e.html}
\date{\today}



\maketitle

\begin{abstract}
{The discontinuous Galerkin (DG) time-stepping method applied to abstract
  evolution equation of parabolic type is studied using a variational
  approach. We establish the inf-sup condition or
  Babu{\v s}ka--Brezzi condition for the DG
  bilinear form. Then, a nearly best approximation property and a
  nearly symmetric error estimate are obtained as corollaries. Moreover, the
  optimal order error estimates under appropriate regularity assumption
  on the solution are derived as direct applications of the standard
  interpolation error estimates. Our method of analysis is new \ns{for the DG time-stepping method; it differs from previous works by which the
  method is formulated as the one-step method}. We apply our abstract
  results to the finite element approximation of \ns{a second order parabolic equation with 
space-time variable coefficient functions} in a polyhedral domain, and derive the optimal order error estimates in several norms. 
}
\end{abstract}

\section{Introduction}
\label{sec:intro}

The discontinuous Galerkin (DG) time-stepping method, which is
designated below as the \emph{dG$(q)$ method}, is a time discretization
method using piecewise polynomials of degree $q$ with an integer $q\ge
0$. The dG$(q)$ method was proposed originally by \cite{lr74} for ordinary differential equations.
Galerkin time-stepping methods of other kinds were proposed earlier by
\cite{hul72b,hul72a}. 
In those works, the methods were formulated using continuous test functions and discontinuous trial functions; the  dG$(q)$ method uses discontinuous functions for both test and trial functions.   
Later, the
method was applied to space-time discretization method for the moving
boundary problem of the heat equation by \cite{jam78}.
Standard time-discretization methods are formulated as one-step or
multi-step methods: approximations are computed at nodal
points. By contrast, the dG$(q)$ method gives approximations as
piecewise polynomials so that approximations at arbitrary \ns{points} are available.
Therefore, the method is useful to address moving boundary
problems and a system composed of
equations having different natures. Indeed, the method is applied actively
to fluid--structure
interaction problems (see \cite{btt13}). 

It was described in \cite{lr74} that the dG$(q)$ method is
interpreted as an one-step method and that it is strongly A-stable of order
$2q+1$. Moreover, after applying a numerical
quadrature formula, the dG$(q)$ method to $y'(t)=\lambda y(t)$ with a
scalar $\lambda$ was found to agree with the sub-diagonal $(q+1,q)$ Pad{\' e} rational
approximation of $e^{-z}$ (see \cite{lr74}). 
Particularly, the dG$(0)$ method implies the backward Euler method. 
For this reason, earlier studies of stability and convergence of the
dG$(q)$ method are accomplished by formulating the method as a
one-step method. However, this seems to make
analysis somewhat intricate, especially 
\ns{when the target equation has variable coefficients.} 

The purpose of this paper is to present a different approach: 
we study the $\textup{dG}(q)$ method using a variational approach 
\ns{based on the formulation of J. L. Lions}. 
To clarify the variational characteristics of the $\textup{dG}(q)$ method, 
we apply the method to abstract evolution equations of parabolic type 
(the coefficient might depend on the time). 
Then, the finite element approximation of a general second order 
parabolic equation is studied as an application of abstract results. 
The $\textup{dG}(q)$ method is the Galerkin approximation of the 
variational formulation of the equation and several techniques 
developed in the literature of the DG method are applicable. 
\ns{For example, we introduce several DG norms which depend on 
the time partitions and apply them to establish the (time) discrete 
inf-sup condition which is one of the main results of this paper. 
The discrete inf-sup condition plays a crucial role for analysis 
of the $\textup{dG}(q)$ method. We also use the (time) continuous 
inf-sup conditions for the analysis of the finite element approximation.} 
Consequently, the analysis becomes greatly simplified for 
\ns{problems with space-time variable coefficients and} for any $q$. 
\ns{Application of the inf-sup condition to numerical analysis for 
parabolic problems is not originally our idea. Many researchers 
have utilized the similar conditions for various problems: \cite{brk17}, \cite{lm17}, \cite{ss09}, \cite{and13}, \cite{up14}, \cite{ns19} and so on. 
However, we are not aware of any previous systematic use of the 
discrete and continuous inf-sup conditions for deducing nearly 
best approximation results together with optimal order error 
estimates for the $\textup{dG}(q)$ method and the finite 
element version of the $\textup{dG}(q)$ method.}

We first formulate the problem to be addressed.
Letting $H$ and $V$ be (real) Hilbert spaces such that $V\subset H$ are dense
with the continuous injection, then the inner product and norms are denoted
as $(\cdot,\cdot)=(\cdot,\cdot)_H$,
$(\cdot,\cdot)_V$, $\|\cdot\|=\|\cdot\|_H$ and $\|\cdot\|_V$. The topological dual spaces
$H$ and $V$ are denoted respectively as $H'$ and $V'$.
As usual, we identify $H$ with $H'$ and consider the triple $V\subset H\subset V'$. Moreover,
$\dual{\cdot,\cdot}=\mbox{}_{V'}\dual{\cdot,\cdot}_{V}$ represents the
duality pairing between $V'$ and $V$. 
Let $J=(0,T)$ with $T>0$.

Assuming that, for a.e.~$t\in J$, we are given a linear operator $A(t)$ of $V\to V'$, and assuming that there exist two positive constants
$M$ and $\alpha$ which are independent of $t\in J$ such that
\begin{subequations}
\begin{align}
 |\dual{A(t)w,v}|&\le M\|w\|_V\|v\|_V && (w,v\in V,~\mbox{a.e.~}t\in J),\label{eq:0a}\\
 \dual{A(t)v,v} &\ge \alpha \|v\|_V^2 && (v\in V,~\mbox{a.e.~}t\in J),\label{eq:0b}
\end{align}
\label{eq:00}
\end{subequations}
we consider the abstract evolution equation of parabolic type as 
\begin{equation}
u'+A(t)u=F(t),\quad t\in J;\qquad u(0)=u_0,
\label{eq:0}
\end{equation}
where $u'$ denotes $du(t)/dt$ and where $F:J\to V'$ and $u_0\in H$ are given
functions. \ns{Note that $\alpha$ and $M$ may depend on the final time $T$.}

Several frameworks and methods can be used to establish the well-posedness
(unique existence of a solution) of \eqref{eq:0}. For this study, we use the variational method of J.~L. Lions. 
To recall it, we require some additional notation. 
\ns{We use the Bochner spaces $L^q(J;U)$, $1\le q\le \infty$, and the Bochner--Sobolev spaces $H^m(J;U)$, $m\ge 1$, for a Banach space $U$. Their norm are defined as 
\begin{gather*}
\|v\|_{L^q(J;U)}^q = 
\int_J \|v(t)\|_U^q~dt \quad (1\le q<\infty),\qquad
\|v\|_{L^\infty(J;U)}=\sup_{t\in J}\|v(t)\|_U \quad (q=\infty),\\
\|v\|_{H^m(J;U)}^2 =\sum_{l=0}^m \|(d/dt)^l v\|_{L^2(J;U)}^2. 
\end{gather*}
}
Set 
\begin{align*}
& \mathcal{X}=L^2(J;V)\cap H^1(J;V'), && \|w\|_{\mathcal{X}}^2=\|w\|_{L^2(J;V)}^2+\|w'\|_{L^2(J;V')}^2,\\
& \mathcal{Y}_1=L^2(J;V),&&\|v\|_{\mathcal{Y}_1}^2 =\|v\|_{L^2(J;V)}^2,\\  
& \mathcal{Y}=\mathcal{Y}_1 \times H,&& \|(v_1,v_2)\|_{\mathcal{Y}}^2 =\|v_1\|_{L^2(J;V)}^2+\|v_2\|^2 . 
\end{align*}
The weak formulation of \eqref{eq:0} is stated as follows. Given
\begin{subequations} 
\label{eq:1}
\begin{equation}
 F\in L^2(J;V'),\qquad u_0\in H,
 \label{eq:data}
\end{equation}
find $u\in \mathcal{X}$ such that
\begin{equation}
B(u,v)=\int_J\dual{F,v_1}~dt
 + (u_0,v_2)\qquad (\forall v=(v_1,v_2)\in\mathcal{Y}),
\label{eq:1a}
\end{equation}
where
\begin{equation}
 B(w,v)=
 \int_J \left[
	 \dual{w',v_1}+\dual{A(t)w,v_1}\right]~dt+(w(0),v_2)
\label{eq:1b}
\end{equation}
for $w\in\mathcal{X}$ and $v=(v_1,v_2)\in\mathcal{Y}$.
\end{subequations}

The space $\mathcal{X}$ is embedded continuously in the set of
$H$-valued continuous functions on $\overline{J}$ (see Theorem XVIII-1 of \cite{dl92} \ns{for example}).  
In other words, there exists a positive constant $C_{\mathrm{Tr},T}$ depending only on $T$ such that 
  \begin{equation}
\max_{t\in \overline{J}}\|v(t)\|_H \le C_{\mathrm{Tr},T}\|v\|_{\mathcal{X}}\qquad
 (v\in \mathcal{X}).
  \label{eq:ch}
  \end{equation}
   Particularly, $w(0)\in H$ in \eqref{eq:1b} is
   well-defined. 

The bilinear form $B$ is bounded in $\mathcal{X}\times \mathcal{Y}$ as
\[
 \|B\|=\sup_{w\in\mathcal{X},v\in\mathcal{Y}}\frac{|B(w,v)|}{\|w\|_\mathcal{X}\|v\|_{\mathcal{Y}}}<\infty.
\]
Moreover, it is known that (see Theorem 6.6 of \cite{eg04}):  
\begin{subequations} 
 \label{eq:bnb}
\begin{gather}
\exists \beta>0,\quad \inf_{w\in\mathcal{X}}\sup_{v\in\mathcal{Y}}\frac{B(w,v)}{\|w\|_{\mathcal{X}}\|v\|_{\mathcal{Y}}}= \beta ; \label{eq:bnb1}\\
v\in \mathcal{Y},\quad (\forall w\in\mathcal{X},\ B(w,v)=0)\quad
 \Longrightarrow \quad (v=0), \label{eq:bnb2}
\end{gather}
\end{subequations}
\ns{where $\beta$ depends only on $\alpha$, $M$ and $C_{\textup{Tr},T}$.} 
Therefore, we can apply the Banach--Ne{\v c}as--Babu{\v s}ka theorem or
Babu{\v s}ka--Lax--Milgram theorem (see Theorem 2.6 of \cite{eg04} \ns{or} Theorem 5.2.1 of \cite{ba72} for example) to conclude that 
there exists a unique
$u\in\mathcal{X}$ satisfying \eqref{eq:1} and it satisfies
\[
 \|u\|_\mathcal{X}\le \ns{\frac{1}{\beta}}\left(\|F\|_{L^2(J;V')}+\|u_0\|\right).
\]
Actually, \eqref{eq:bnb} is the necessary and
sufficient condition for the well-posedness of \eqref{eq:1}. (The case $u_0=0$ is described explicitly
in \cite{eg04}. However, the modification to the case $u_0\ne 0$
is straightforward.) 
Equality \eqref{eq:bnb1} is commonly 
designated as the \emph{inf-sup condition} or \emph{Babu{\v s}ka--Brezzi
condition}.   
Furthermore, \eqref{eq:bnb} is equivalent to 
\begin{equation} 
 \label{eq:bnb3}
  \exists \beta>0,\quad
   \inf_{w\in\mathcal{X}}\sup_{v\in\mathcal{Y}}\frac{B(w,v)}{\|w\|_{\mathcal{X}}\|v\|_{\mathcal{Y}}}=
   \inf_{v\in\mathcal{Y}}\sup_{w\in\mathcal{X}}\frac{B(w,v)}{\|w\|_{\mathcal{X}}\|v\|_{\mathcal{Y}}}=
   \beta. 
\end{equation}
This equivalence is verified by considering the associating operators
with $B$ and the operator norms of their inverse operators.    
Indeed, this equivalence plays an important role in the discussion below.
 
The dG$(q)$ method described below (see \eqref{eq:2}) is based on
the formulation \eqref{eq:1}, which means that the dG$(q)$ method is
consistent with \eqref{eq:1} in the sense of Lemma \ref{la:2}. 
(The consistency is also called the Galerkin orthogonality.)
Therefore, it is natural to ask whether a discrete version of
\eqref{eq:bnb}, particularly \eqref{eq:bnb1}, is available. If it is established, then the
best approximation property and optimal order error estimates are
obtained as direct consequences.  

In this paper, after describing the dG$(q)$ method, we first
prove that there exists a positive constant
$c_1$ such that (see Theorem \ref{th:0}) 
\[
      \inf_{w_\tau\in\mathcal{S}_\tau}
       \sup_{v_\tau\in\mathcal{S}_\tau}\frac{B_\tau(w_\tau,v_\tau)}{\normx{w_\tau}\normyb{v_\tau}}=
       c_1
\]
which is a discrete version of \eqref{eq:bnb1}. Herein, $B_\tau$ is the DG
approximation of $B$. Also, $\mathcal{S}_\tau$ is the set of $V$-valued
piecewise polynomials of degree $q$ defined on a non-uniform partition of $J$
with size parameter $\tau>0$. Moreover, $\normx{\cdot}$ and $\normyb{\cdot}$
are the DG norms corresponding to $\|\cdot\|_{\mathcal{X}}$ and
$\|\cdot\|_{\mathcal{Y}}$, respectively. (The precise definition of these
symbols will be presented in Section \ref{sec:2}.)  
\ns{The DG norm $\normx{\cdot}$ is a simple discrete counterpart of the $\mathcal{X}$ norm, while the definition of $\normyb{\cdot}$ is somewhat strange.  
Actually, it is crucial to introduce $\normyb{\cdot}$ and 
a special projection operator $\pi_h$ defined as \eqref{eq:71} in deriving the discrete inf-sup condition above. 
}
Then, as a direct
consequence, we demonstrate that there exists a
positive constant $c_1'$ such that (see Theorem \ref{th:1}) 
\[
 \normx{u-u_\tau}\le c_1'\inf_{w_\tau\in\mathcal{S}_\tau}\normxa{u-w_\tau},  
\]
where \ns{$u$ and} $u_\tau\in\mathcal{S}_\tau$ respectively represent the solutions of \ns{\eqref{eq:1} and} the dG$(q)$ method. $\normxa{\cdot}$ denotes another DG norm corresponding to
$\|\cdot\|_{\mathcal{X}}$ satisfying $\normx{v_\tau}\le \normxa{v_\tau}$
for $v_\tau\in\mathcal{S}_\tau$. This result is neither a best
approximation property 
nor a symmetric error estimate in the sense of \cite{dl02}; it is only a
nearly best approximation property and nearly symmetric error
estimate. 
\ns{
However, the difference between $\normx{\cdot}$ and $\normxa{\cdot}$ is insignificant. In fact, using this result, together with 
the standard interpolation error estimates (see Lemma \ref{la:20}),}   
one can obtain optimal order error
estimates under appropriate regularity assumptions on $u$. \ns{That is, }we prove (see Theorem \ref{th:2})
\begin{equation*}
 \left(\sum_{n=0}^{N-1}\|u'-u_\tau'\|_{L^2(J_n;V')}^{2}\right)^{1/2} 
\ns{+\|u-u_\tau\|_{L^2(J;V)}}
\le c_3
  \tau^{q}
  \|u^{(q)}\|_{\mathcal{X}}
\end{equation*}
and 
\[
  \sup_{1\le n\le N}\|u(t_n)-u_\tau(t_n)\|+\|u-u_\tau\|_{L^2(J;V)}
  \le  c_4\tau^{q+1}\|u^{(q+1)}\|_{\mathcal{Y}_1}.
\]

Application of those results to concrete partial
differential equations is straightforward. As an illustration, we
consider the finite element approximation \ns{of a second-order parabolic equation \eqref{eq:h}} 
in a polyhedral domain $\Omega\subset\mathbb{R}^d$, $d=2,3$. 
\ns{The coefficient functions may depend on both space and time variables, while most of the previous studies dealt only with the heat equation.   
The space variable is
discretized by 
the conforming $\mathrm{P}^k$ finite element method, which is
designated below as the \emph{cG$(k)$ method}. We recall that the semi-discretization error $u-u_h$ is well studied so far, where $u$ and $u_h$ respectively represent the solutions of the parabolic equation and the resulting space semi-discrete equation by the cG$(k)$ method. 
However, we state (see Lemma \ref{la:22}) a version of the proof based on continuous inf-sup conditions for completeness. 
The dG$(q)$ method is applied to the space semi-discrete equation, which is designated below as the 
\emph{dG$(q)$cG$(k)$ method}. 
To apply the error estimates above, we have to study the stability for $\partial_t^qu_h$. 
As an application of the continuous inf-sup condition, we derive (see Lemma \ref{la:23}) 
 \begin{equation*}
     \|\partial_t^{q+1}u_h\|_{L^2(J;X_h')}^2 + 
\|u_h\|_{H^{q}(J;H^1_0(\Omega))}^2
      \le           C\left(\|\partial_t^{q+1}u\|_{L^2(J;H^{-1}(\Omega))}^2+
 \|u\|_{H^{q}(J;H^1_0(\Omega))}^2\right)    ,
\end{equation*}
where $X_h$ stands for the conforming $\mathrm{P}^k$ finite element space (the meaning of $X_h'$ should be paid a special attention. See Section \ref{sec:7}).  
Consequently, one of our final error estimates reads (see Theorem \ref{th:20})
\[
\sup_{1\le n\le N}\|u(t_n)-u_{h,\tau}(t_n)\|_{L^2(\Omega)}+ \|u-u_{h,\tau}\|_{L^2(J;H^1_0(\Omega))}\le c_6(h^k+\tau^{q+1})\|u\|_{{Z}_2},
\]
where $u_{h,\tau}$ denotes the solution of the $\mathrm{dG}(q)\mathrm{cG}(k)$ method and the quantity $\|u\|_{Z_2}$ contains norms of higher-order derivatives of $u$. 
The assumption $\|u\|_{Z_2}<\infty$ is optimal for deducing the convergence order $O(h^k+\tau^{q+1})$ in view of the standard interpolation error estimates. We also derive the optimal order error estimates for $\partial_tu_{h,\tau}$. In deriving these error estimates,  
time partition and space triangulation need not to be quasi-uniform. Moreover, our estimates include no log-factors and valid for any $q\ge 0$ and $k\ge 1$. 
The salient benefit of our analysis based on the continuous and discrete inf-sup conditions is its applicability for a wide range of problems in a coherent manner. In fact, we perform analysis for general equations with variable coefficients in almost the same way as that for the heat equation.
}

At this stage, let us briefly review previous studies of convergence of the
dG$(q)$ method for parabolic equations to clarify the novelty of our results.
As described above, \cite{jam78} studied space-time finite
element discretization for the heat equation defined in the moving
domain $Q_T=\cup_{t\in J}\{t\}\times \Omega(t)$; the space-time slab
$Q_n=\cup_{t\in (t_n,t_{n+1})}\{t\}\times \Omega(t)$ is discretized directly using a space-time
mesh with the size parameter $\rho$ and the conforming $\mathrm{P}^k$
finite element space on the mesh is considered. He proved $O(\rho^k)$
convergence results in the $L^2$ norm for time and $H^1$ norm for
space. $O(\rho^k)$ convergence results in the $L^\infty$ norm for time and $L^2$ norm for
space were also reported.
\cite{ejt85} described consideration of the abstract evolution equation
of the form \eqref{eq:1} in a Hilbert space $H$, where $A$ is assumed to
be independent of $t$, self-adjoint on $H$ and positive-definite with compact
inverse $A^{-1}$. The optimal convergence, say $O(\tau^{q+1})$
convergence, in the $L^\infty(J;H)$ and a super-convergence, say $O(\tau^{2q+1})$
convergence, are presented at nodal points in the $H$ norm.
For super-convergence, the initial value $u_0$ is expected to satisfy additional
boundary conditions as well as regularity conditions. They succeeded in
relaxing those conditions, but the resulting error estimates include the
log-factor of the form $\log(1/\tau)$. Applications to the finite
element method for the heat equation in a fixed smooth domain
$\Omega\subset\mathbb{R}^d$ are also discussed. They {studied the $\mathrm{dG}(q)\mathrm{cG}(k)$ method and} offered
optimal convergence results in the $L^\infty(J;L^2(\Omega))$ norm and
super-convergence results with the log-factor at time-nodal points.  
Those results were
extended to several directions in \cite[Chapter 12]{tho06}. 
\cite{ej91,ej95} examined adaptive algorithms
and a posteriori error
estimates for the $\mathrm{dG}(q)\mathrm{cG}(1)$ method for the heat
equation with $q=0,1$. They also proved several a priori estimates of
optimal order with the log-factor in, for example, the
$L^\infty(J;L^2(\Omega))$ and $L^\infty(J\times \Omega)$ norms.  
The $hp$-version of the dG$(q)$ method was studied carefully in \cite{ss00}. In particular, the exponential convergence results in the
$L^2(J;V)$ norm were proved using the graded time partitions. 
\cite{cw06} considered the
$\mathrm{dG}(q)\mathrm{cG}(k)$ method for \ns{a general parabolic equation and presented sharp error estimates using space-time local projection operators.}　In \cite{ej91}, \cite{ej95}, and \cite{cw06}, the finite element
spaces might be different at each time slab. \ns{\cite{mn06} studied a posteriori error estimates for both linear and nonlinear problems using 
the parabolic reconstruction operator.}  
\cite{lv16} proved a best approximation property
of the form
\[
 \|u-u_{h,\tau}\|_{L^\infty(J\times \Omega)}\le C\log (T/\tau) \log
 (1/h) \inf_{\chi\in X_{h,\tau}}\|u-\chi\|_{L^\infty(J\times\Omega)},
\]
where $u_{h,\tau}$ denotes the solution of the
$\mathrm{dG}(q)\mathrm{cG}(k)$ method for the heat equation in a convex
polyhedral domain $\Omega\subset\mathbb{R}^d$, $d=2,3$, and $X_{h,\tau}$
is the $\mathrm{dG}(q)\mathrm{cG}(k)$ finite element space. Therein, the
quasi-uniformity conditions were assumed both for time and space meshes.

This paper comprises four sections with an appendix. In Section
\ref{sec:2}, the dG$(q)$ method for \eqref{eq:1} and the main
results, Theorems \ref{th:0}--\ref{th:2}, are stated. The proof of
Theorems \ref{th:1} and \ref{th:2} are also described there. The proof
of Theorem \ref{th:0} is presented in Section \ref{sec:5}. Section
\ref{sec:7} is devoted to the study of the dG$(q)$cG$(k)$ method for 
\ns{a general second order parabolic equation with space-time variable coefficient functions. 
Stability and convergence results for the $\textup{cG}(k)$ method and the optimal order error estimates for the $\textup{dG}(q)\textup{cG}(k)$ method are presented there. 
The proof of some auxiliary results, the continuous inf-sup conditions, is given in Appendix \ref{sec:a}. }


\section{DG time-stepping method dG$(q)$ and main results}
\label{sec:2}

Let $N$ be a positive integer. 
We introduce $N+1$ distinct points $0=t_0<t_1<\cdots<t_n<\cdots<t_N=T$. Set $J_n=(t_n,t_{n+1}]$ and
$\tau_n=t_{n+1}-t_n$ for $n=0,\ldots,N-1$.

We consider the partitions of $J$ as $\Delta_\tau=\{J_0,\ldots,J_{N-1}\}$, where
$\displaystyle{\tau=\max_{0\le n\le N-1}\tau_n}$. Without loss of generality,
we assume that $\tau\le 1$. We set
\begin{equation}
C^0(\Delta_\tau;H)=\{v\in L^\infty(J;H) \mid v|_{J_n}\in C^0(J_n;H),\ 0\le n\le N-1\},
  \label{eq:C}
\end{equation}
where $C^0(J_n;H)$ denotes the set of $H$-valued continuous functions on
$J_n$. Spaces $C^0(\Delta_\tau;V)$ and $C^0(J_n;V)$ are defined
similarly. For arbitrary $v\in C^0(\Delta_\tau;H)$, we write
\begin{equation}
 v^{n,+}=\lim_{t\downarrow t_n}v(t),\quad v^{n+1}=v(t_{n+1})\qquad (n=0,\ldots,N-1).
  \label{eq:pm}
\end{equation}
Let $q\ge 0$ be an integer and set  
\begin{equation}
 \mathcal{S}_\tau=\mathcal{S}_\tau^q(H,V)=\{v\in C^0(\Delta_\tau;H) \mid
  v|_{J_n}\in\mathcal{P}^q(J_n;V),\ 0\le n\le N-1\}, 
  \label{eq:S}
\end{equation}
where $\mathcal{P}^q(J_k;V)$ denotes the set of $V$-valued polynomials
of $t\in J_n$ with degree $\le q$.

The DG time-stepping method dG$(q)$ is presented below.
Find $u_\tau\in\mathcal{S}_\tau$ such that
\begin{subequations} %
\label{eq:2}
\begin{equation}
B_\tau(u_\tau,v_\tau)=
 \int_J\dual{F,v_\tau}~dt
 + (u_0,v^{0,+}_\tau)\qquad (\forall v_\tau\in\mathcal{S}_\tau),
\label{eq:2a}
 \end{equation}
where
\begin{equation}
B_{\tau}(u_\tau,v_\tau)=  \sum_{n=0}^{N-1}\int_{J_n} \left[
	 (u_\tau',v_\tau)+\dual{A(t)u_\tau
	 ,v_\tau}\right]~dt+
 (u_\tau^{0,+},v_\tau^{0,+})
+\sum_{n=1}^{N-1}(u_\tau^{n,+}-u_\tau^n,v_\tau^{n,+}) .
 \label{eq:2b} 
\end{equation}
\end{subequations}
\ns{Applying integration by parts to the term $\int_{J_n}(u_\tau',v_\tau)~dt$, we obtain an} 
alternate expression
\begin{equation}
B_{\tau}(u_\tau,v_\tau)=  \sum_{n=0}^{N-1}\int_{J_n} \left[
	 -(u_\tau,v_\tau')+\dual{A(t)u_\tau
	 ,v_\tau}\right]~dt+
 (u_\tau^{N},v_\tau^{N})
+\sum_{n=1}^{N-1}(u_\tau^{n},v_\tau^{n}-v_\tau^{n,+}).
 \label{eq:2c} 
\end{equation}

Because the solution $u$ of \eqref{eq:1} is a function of
$C^0(\overline{J};H)$, we have $(u^{n,+}-u^n,v_\tau^{n,+})=0$ for $v_\tau\in
\mathcal{S}_\tau$. Consequently, we obtain the following result.   
\begin{lemma}[Consistency]
If $u\in\mathcal{X}$ and $u_\tau\in\mathcal{S}_\tau$ respectively represent the solutions of \eqref{eq:1} and \eqref{eq:2}, then we have 
$B_\tau(u-u_\tau,v_\tau)=0$ for any $v_\tau\in \mathcal{S}_\tau$.
\label{la:2}
\end{lemma}

For $v_\tau\in\mathcal{S}_\tau$ and a sequence $\{k_n\}=\{k_n\}_{n=0}^{N-1}$,
we set the following.
\begin{align*}
\nu(v_\tau,\{k_n\})&= 
\sum_{n=0}^{N-1}\int_{J_n} \left[\|v_\tau'\|_{V'}^2+\|v_\tau\|_V^2\right]~dt+\|v_\tau^{0,+}\|^2+\sum_{n=1}^{N-1}k_n\|v_\tau^{n,+}-v_\tau^{n}\|^2,\\
\eta(v_\tau,\{k_n\})&= 
\sum_{n=0}^{N-1}\int_{J_n}
 \|v_\tau\|_V^2~dt+\|v_\tau^{0,+}\|^2+\sum_{n=1}^{N-1}k_n\|v_\tau^{n,+}\|^2,\\
\nu^*(v_\tau,\{k_n\})&= 
\sum_{n=0}^{N-1}\int_{J_n}
 \left[\|v_\tau'\|_{V'}^2+\|v_\tau\|_V^2\right]~dt+\|v_\tau^{N}\|^2
 +\sum_{n=1}^{N-1}k_n\|v_\tau^{n,+}-v_\tau^{n}\|^2,\\
 \eta^*(v_\tau,\{k_n\})&= 
\sum_{n=0}^{N-1}\int_{J_n}
 \|v_\tau\|_V^2~dt+\|v_\tau^{N}\|^2+\sum_{n=1}^{N-1}k_n\|v_\tau^{n,+}\|^2. 
\end{align*}

We use the following norms:  
\begin{subequations} 
\label{eq:11}
\begin{align}
&\normx{v_\tau}^2=\nu(v_\tau;\{1\});\quad 
&&\normxa{v_\tau}^2=\nu(v_\tau;\{\tau_n^{-1}\});\label{eq:11a}\\
&\normy{v_\tau}^2=\eta(v_\tau;\{1\}) ;\quad 
&&\normyb{v_\tau}^2=\eta(v_\tau;\{\tau_n\}); \label{eq:11b}\\
&\normxx{v_\tau}^2=\nu^*(v_\tau;\{1\}); \quad 
&&\normxxa{v_\tau}^2=\nu^*(v_\tau;\{\tau_n^{-1}\}); \label{eq:11c}\\
&\normyy{v_\tau}^2=\eta^*(v_\tau;\{1\}); \quad 
&&\normyyb{v_\tau}^2=\eta^*(v_\tau;\{\tau_n\}). \label{eq:11d}
\end{align}
\end{subequations}

Because we are assuming $\tau_n\le 1$, we have 
$\normx{v_\tau}\le \normxa{v_\tau}$ and  
$\normyb{v_\tau}\le \normy{v_\tau}$ for $v_\tau\in\mathcal{S}_\tau$. 
The same relations hold for other norms defined
  as \eqref{eq:11c} and \eqref{eq:11d}. 
The following lemma directly follows those definitions. 

\begin{lemma}
 \label{la:8}
 For $w_\tau,v_\tau\in \mathcal{S}_\tau$, we have
 \begin{equation*}
  |B_\tau(w_\tau,v_\tau)|
   \le
\begin{cases}
M_0 \normx{w_\tau}\normy{v_\tau}    &\\
M_1\normxa{w_\tau}\normyb{v_\tau}  &\\
M_2\normyy{w_\tau}\normxx{v_\tau}  &
\end{cases}
 \end{equation*}
where $M_j$, $j=0,1,2$, are positive constants depending only on $M$. 
\end{lemma}

We are now in a position to state the main results presented in this paper. 
  
  \begin{theorem}
There exist positive constants $c_1$ and $c_2$ depending only on
  $\alpha$, $M$ and $q$ such that
   \begin{subequations} 
    \label{eq:16z}
     \begin{align}
      \inf_{w_\tau\in\mathcal{S}_\tau}
       \sup_{v_\tau\in\mathcal{S}_\tau}\frac{B_\tau(w_\tau,v_\tau)}{\normx{w_\tau}\normyb{v_\tau}}= c_1; 
 \label{eq:16}\\
      \inf_{w_\tau\in\mathcal{S}_\tau}
       \sup_{v_\tau\in\mathcal{S}_\tau}\frac{B_\tau(w_\tau,v_\tau)}{\normyyb{w_\tau}\normxx{v_\tau}}= c_2.
  \label{eq:23z}
     \end{align} 
   \end{subequations} 
  \label{th:0}
  \end{theorem}

The proof of Theorem \ref{th:0} will be given in Section \ref{sec:5}. 
The following result, Theorem \ref{th:1}, is a readily
obtainable corollary of Theorem \ref{th:0} and Lemma \ref{la:8}.

 \begin{theorem}
\label{th:1}
Letting $u\in\mathcal{X}$ and $u_\tau\in\mathcal{S}_\tau$ respectively denote the solutions of \eqref{eq:1} and \eqref{eq:2}, 
we have  
\begin{subequations} %
\label{eq:100}
\begin{align}
\normx{u-u_\tau}&\le \left(1+\frac{M_1}{c_1}\right)\inf_{w_\tau\in\mathcal{S}_\tau}\normxa{u-w_\tau},\label{eq:100a}  \\   
\normyyb{u-u_\tau}&\le \left(1+\frac{M_2}{c_2}\right)\inf_{w_\tau\in\mathcal{S}_\tau}\normyy{u-w_\tau},\label{eq:100b}   
\end{align}
 \end{subequations}
 \end{theorem}

 \begin{proof}
  Let $w_\tau\in\mathcal{S}_\tau$ be arbitrary. In view of Theorem
  \ref{th:0}, Lemmas
  \ref{la:2} and \ref{la:8}, we have
  \begin{align*}
\|w_\tau-u_\tau\|_{\mathcal{X},\tau}
    & \le \frac{1}{c_1}\sup_{v_\tau\in
    \mathcal{S}_\tau}\frac{B_\tau(w_\tau-u_\tau,v_\tau)}{\|v_\tau\|_{\mathcal{Y},\tau,\#}} 
    =\frac{1}{c_1}\sup_{v_\tau\in
    \mathcal{S}_\tau}\frac{B_\tau(w_\tau-u,v_\tau)}{\|v_\tau\|_{\mathcal{Y},\tau,\#}} \\
    & \le \frac{M_1}{c_1}\normxa{w_\tau-u}. 
   \end{align*}
Therefore, using the triangle inequality, we obtain \eqref{eq:100a}. The
  proof of \eqref{eq:100b} is the same.   
 \end{proof}
 
We derive some optimal order error estimates using Theorem
\ref{th:1}. 
We set
$t_{n,j}=t_{n}+j(t_{n+1}-t_n)/q=t_{n}+j\tau_{n}/q$
$(j=0,\ldots,q)$ \ns{if $q\ge 1$, and $t_{n,0}=t_n$ if $q=0$}. For $v\in L^2(J_n;V)\cap H^1(J_n;V')$, there exists a
unique $I_{n}v\in \mathcal{P}^q(\overline{J_n};V)$ such that
\[
 (I_{n}v)(t_{n,j})=v(t_{n,j})\qquad (j=0,\ldots,q).
\]
The following error 
estimates for the
Lagrange interpolation $I_nv$ is proved in the standard way using
Taylor's theorem (see Theorem 4.A of \cite{zei86} for example). We write
 $v^{(s)}=(d/dt)^{s}v$ for positive integer $s$.
 
\begin{lemma}
 \label{la:20}
Letting $q\ge 1$ be an integer, then there exists an absolute positive constant
 $C$ such that 
 \begin{subequations} %
  \label{eq:107}
\begin{align}
\|v-I_n v\|_{L^2(J_n;U)} & \le C\tau_n^{s}\|v^{(s)}\|_{L^2(J_n;U)},\label{eq:107a}\\
\|v'-(I_n v)'\|_{L^2(J_n;U)} & \le C\tau_n^{s-1}\|v^{(s)}\|_{L^2(J_n;U)}\label{eq:107b}
\end{align}
 \end{subequations} %
for an integer $s$ with $1<s\le q+1$ and $v\in H^{q+1}(J_n;U)$, where $U=V,H,V'$. Constant $C$ is independent of $U$. 
Moreover, the estimate \eqref{eq:107a} remains true for $q=0$ and $s=1$. 
\end{lemma}

Finally we can state the following optimal order error estimates. 

\begin{theorem}
 \label{th:2}
Letting $u\in\mathcal{X}$ and $u_\tau\in\mathcal{S}_\tau$ respectively denote the
 solutions of \eqref{eq:1} and \eqref{eq:2}, then if $q\ge 0$ and  
 $u^{(q)}\in\mathcal{X}=L^2(J;V)\cap H^1(J;V')$, it follows that 
\begin{subequations} %
\label{eq:101}
\begin{equation}
 \left(\sum_{n=0}^{N-1}\|u'-u_\tau'\|_{L^2(J_n;V')}^{2}\right)^{1/2} 
\ns{+\|u-u_\tau\|_{L^2(J;V)}}\le c_3
  \tau^{q}
  \|u^{(q)}\|_{\mathcal{X}}.\label{eq:101a}   
\end{equation}
Moreover, if $q\ge 1$ and $u^{(q+1)}\in \mathcal{Y}_1=L^2(J;V)$, then \ns{ 
 \begin{equation}
  \sup_{1\le n\le N}\|u(t_n)-u_\tau(t_n)\|+\|u-u_\tau\|_{L^2(J;V)}
  \le  c_4\tau^{q+1}\|u^{(q+1)}\|_{\mathcal{Y}_1}.
\label{eq:101b}
\end{equation}
}
\end{subequations}
If $q=0$ and $u'\in \mathcal{X}$, the estimate \eqref{eq:101b} holds with replacement
 $\|u^{(q+1)}\|_{\mathcal{Y}_1}$ by
 $\|u'\|_{\mathcal{X}}$. 
Therein, $c_3$ and $c_4$ denote positive constants depending only on $\alpha$,
 $M$ and $q$. 
\end{theorem}

 \begin{proof}
First, we prove \eqref{eq:101b} with $q\ge 1$.  
Let us define $w_\tau\in\mathcal{S}_\tau$ by $w_\tau|_{J_n}=I_nu$ for $n=0,1,\ldots,N-1$. Set $c_2'=1+M_2/c_2$.    
Because $u^{n,+}-w_\tau^{n,+}=0$ for $n=1,\ldots,N-1$ and $u^{N}-w_\tau^{N}=0$, we have by \eqref{eq:100b} and \eqref{eq:107a}
\begin{multline}
 \|u-u_\tau\|_{L^2(J;V)}^2+ \|u(t_N)-u_\tau(t_N)\|^2
+\sum_{n=1}^{N-1}\tau_n\|u_\tau(t_n)-u_\tau^{n,+}\|^2\\
 \le
C^2(c_2')^2\sum_{n=0}^{N-1}\tau_n^{2(q+1)}\|u^{(q+1)}\|_{L^2(J_n;V)}^2,
\label{eq:201}
\end{multline}
which implies \eqref{eq:101b}  with $q\ge 1$. We next consider the case $q=0$. 
Then, \ns{$I_nv$ is a constant function in $J_n$ such that $(I_nv)(t)=v(t_n)$ for $t\in J_n$. }
We still have $u^{n,+}-w_\tau^{n,+}=0$ for $n=1,\ldots,N-1$;
 however, we have $u^{N}-w_\tau^{N}=u(t_N)-u(t_{N-1})\ne 0$. 
Consequently, the left-hand side of \eqref{eq:201} is estimated from above by 
 \[
 C^2(c_2')^2\sum_{n=0}^{N-1}\tau^2_n\|u'\|_{L^2(J_n;V)}^2+\tau_{N-1}^2\sup_{t\in J}\|u'\|^2.
 \]
Therefore, in view of the trace inequality \eqref{eq:ch}, we obtain \eqref{eq:101b}  with replacement
 $\|u^{(q+1)}\|_{\mathcal{Y}_1}$ by $\|u'\|_{\mathcal{X}}$.  
The inequality \eqref{eq:101a} with $q\ge 1$ is proved similarly using \eqref{eq:100a} and
   \eqref{eq:107}. Finally, 
the estimate \eqref{eq:101a} with $q =0$ implies a trivial inequality, because $u_\tau'=0$ on every $J_n$. 
 \end{proof}

\begin{remark}
If $q=0$, the estimate \eqref{eq:101a} gives no information on the convergence. However, the following observation may be useful. 
For \ns{a given $w_\tau\in\mathcal{S}_\tau$}, there exists a unique $\ns{D_\tau w_\tau}\in\mathcal{S}_\tau$ such that 
\[
 \int_{J_n}(\ns{D_\tau w_\tau},v_\tau)~dt=
 \int_{J_n}(\ns{w_\tau}',v_\tau)~dt +\left(\ns{w_\tau^{n,+}-w_\tau^n},v_\tau^{n,+}\right)
\qquad (v_\tau\in \mathcal{S}_\tau,~n=0,\ldots,N-1).
\]  
\ns{In view of Lemma \ref{la:2}, we see that the solution $u_\tau\in\mathcal{S}_\tau$ of \eqref{eq:2} solves }
\[
 \sum_{n=0}^{N-1}\int_{J_n}\dual{u'-\ns{D_\tau}u_\tau,v_\tau}~dt=
\int_{J}\dual{A(t)(u-u_\tau),v_\tau}~dt
 \qquad
 (\forall v_\tau\in\mathcal{S}_\tau),
\]
where $u_\tau^0=u_0$. 
Therefore, applying \eqref{eq:0a} and \eqref{eq:101b}, we deduce 
 \begin{equation}
\mathcal{E}(u'-\ns{D_\tau}u_\tau)
	\le  \begin{cases}
	      \ns{c_4}M \tau^{q+1}\|u^{(q+1)}\|_{\mathcal{Y}_1} & (q\ge 1)\\
	      \ns{c_4}M \tau\|u'\|_{\mathcal{X}} & (q= 0),
	     \end{cases}
       \label{eq:101z}
\end{equation}
where  
 \[
 \mathcal{E}(w)=\sup_{v_\tau\in\mathcal{S}_\tau}\frac{1}{~\|v_\tau\|_{L^2(J;V)}}
  \sum_{n=0}^{N-1}\int_{J_n}|\dual{w,v_\tau}|~dt\quad (w\in C^0(\Delta_\tau;H)).
 \]
The estimate \eqref{eq:101z} shows that $\ns{D_\tau} u_\tau$ provides an approximation of order $\tau^{q+1}$ for $u'$ with any $q\ge 0$. 
\ns{As a matter of fact, $D_\tau u_\tau$ is closely related with the parabolic reconstruction $\hat{u}_\tau\in\mathcal{S}_\tau^{q+1}(H,V)$ of $u_\tau$ introduced by \cite{mn06}. 
More specifically, $\hat{u}_\tau$ is given as 
\[
  \hat{u}_\tau(t)=\int_{t_n}^t (D_\tau u_\tau) (s)~ds+u_\tau^n\qquad (t\in J_n,n=0,\ldots,N-1).
\]
\cite{mn06} successfully used the parabolic reconstruction operator to deduce a posteriori error estimates for the $\textup{dG}(q)$ method applied to linear and nonlinear problems.  
}
\end{remark}

\ns{
\begin{remark}
\label{rem:00}
Reviewing \eqref{eq:201} in the proof of Theorem \ref{th:2}, we have
\[
 \left(\sum_{n=1}^{N-1}\tau_n\|u_\tau(t_n)-u_\tau^{n,+}\|^2\right)^{1/2}
  \le  c_4\tau^{q+1}\|u^{(q+1)}\|_{\mathcal{Y}_1}.
\] 
This is not an error estimate and is an estimation for the jump of $u_\tau$ at nodes. 
\end{remark}
}

\section{Proof of Theorem \ref{th:0}}
\label{sec:5}

This section is devoted to the proof of Theorem \ref{th:0}.
First, we collect some preliminary results. Throughout this section, we let
$n=0,1,\ldots,N-1$, unless otherwise stated explicitly. 

We let $q\ge 1$ for the time being.  
The following projection is a slight modification of
\cite[(12.9)]{tho06}. 
For $v\in L^2(J_n;V)$, there exists a unique
$\tilde{v}\in \mathcal{P}^q(J_n;V)$ such that 
 \begin{subequations} 
  \label{eq:71}
\begin{gather}
 \tilde{v}^{n,+}=0;\label{eq:71a}\\
 \int_{J_n} [v(t)-\tilde{v}(t)]~(t-t_n)^l~dt=0\quad (l=0,1,\ldots,q-1).
\label{eq:71b}
\end{gather}
 \end{subequations}
Projection $\pi_n:L^2(J_n;V)\to\mathcal{P}^q(J_n;V)$ is defined as
$\tilde{v}=\pi_nv$.    
In fact, $\tilde{v}$ is expressed as
${\tilde{v}=\sum_{l=1}^{q}a_l(t-t_n)^l}$ with $a_1,\ldots,a_q\in V$ in view of \eqref{eq:71a}.
Therefore, \eqref{eq:71b} implies the system of $V$-valued linear
equations for unknowns $a_1,\ldots,a_q$. The number of equations is also
$q$: it suffices to check the uniqueness to verify the
existence of $\tilde{v}$. However, it is a direct consequence of the
following \eqref{eq:9b}. Alternatively, one could follow the same argument as
\cite{tho06} to deduce the uniqueness. Therefore, the projection $\pi_n$ is
well-defined.

\begin{lemma}
 \label{la:6}
Letting $q\ge 1$ and $U=H,V$, the projection $\pi_n$ satisfies the following: 
\begin{subequations} 
\begin{align}
  \int_{J_n}(\pi_n{v},\chi)_U~dt&=\int_{J_n}(v,\chi)_U~dt&&( \chi\in
  \mathcal{P}^{q-1}(J_n; V)); 
 \label{eq:9a}\\
\|\pi_nv\|_{L^2(J_n;U)}&\le
 C_q\|v\|_{L^2(J_n;U)}&& (v\in L^2(J_n;V)),
 \label{eq:9b}
\end{align}
\end{subequations}
where $C_q$ denotes a positive constant depending only on $q$.
\end{lemma}

\begin{proof}[Proof of \eqref{eq:9a}]
Let $v\in L^2(J_n;V)$ and $\chi\in
  \mathcal{P}^{q-1}(J_n;V)$. Writing 
 ${\chi=\sum_{l=0}^{q-1}b_l(t-t_n)^l}$ with $b_0,\ldots,b_{q-1}\in V$ and using \eqref{eq:71b}, we have 
\[
  \int_{J_n}\left(v-\pi_nv,\chi\right)_U~dt
 =\sum_{l=0}^{q-1}\left(\int_{J_n}(v-\pi_nv)(t-t_n)^l~dt,\ b_l\right)_U=0.
  \]
In fact, the first equality is justified because $v:J_n\to U$ is Bochner
 integrable; see \S V.5 of \cite{yos80} for example.
\end{proof}

\begin{proof}[Proof of \eqref{eq:9b}]
Let $v\in L^2(J_n;V)$ and set $\tilde{v}=\pi_nv\in
 \mathcal{P}^q(J_n;V)$. The proof is divided into two steps.  

\noindent \emph{Step 1.} Substituting $\chi=\tilde{v}'\in\mathcal{P}^{q-1}(J_n;V)$ into \eqref{eq:9a},
 we have
  \begin{equation}
  \label{eq:9i}
  \int_{J_n}(\tilde{v},\tilde{v}')_U~dt
 =\int_{J_n}(v,\tilde{v}')_U~dt. 
 \end{equation}
To estimate the right-hand side, we apply the inverse
 inequality
\[
  \|w'\|_{L^2(J_n;U)}\le \frac{c_q}{2\tau_n}\|w\|_{L^2(J_n;U)}\quad (w\in\mathcal{P}^q(J_n;U)),
\]
where $c_q>0$ denotes a constant depending only on $q$. The
 proof of this inequality is exactly the same as the scalar case. In
 particular, $c_q$ is independent of $U$. Using this, we have
\begin{equation}
  \label{eq:9i1}
\left| \int_{J_n}(v,\tilde{v}')_U~dt\right|
 \le 
\|v\|_{L^2(J_n;U)}\|\tilde{v}'\|_{L^2(J_n;U)}
 \le \frac{c_q}{2\tau_n}\|v\|_{L^2(J_n;U)}\|\tilde{v}\|_{L^2(J_n;U)}.
\end{equation}
On the other hand, by virtue of \eqref{eq:71a}, we have 
\begin{equation}
  \label{eq:9i2}
 \int_{J_n}(\tilde{v},\tilde{v}')_U~dt=\frac{1}{2}\int_{J_n}\frac{d}{dt}\|\tilde{v}(t)\|_U^2~dt
 =\frac12\|\tilde{v}(t_{n+1})\|_U^2.
\end{equation}
 
Combining \eqref{eq:9i}, \eqref{eq:9i1} and \eqref{eq:9i2}, we obtain
\[
  \|\tilde{v}(t_{n+1})\|_U^2
  \le \frac{c_q}{\tau_n}\|v\|_{L^2(J_n;U)}\|\tilde{v}\|_{L^2(J_n;U)}.
\]

\noindent \emph{Step 2. }For $\tilde{v}$ defined above and $s\in J_n$,
 there exists a $\hat{v}\in \mathcal{P}^{q}(J_n;V)$ such that 
\begin{subequations} 
\label{eq:9f}
\begin{gather}
  \hat{v}(t_{n+1})=\tilde{v}(t_{n+1});\label{eq:9f0} \\
  \int_{J_n}(\hat{v},\chi)_U~dt=\int_{t_n}^s(\tilde{v},\chi)_U~dt\quad ( \chi\in
  \mathcal{P}^{q-1}(J_n; V)); 
 \label{eq:9f1}\\
\|\hat{v}\|_{L^2(J_n;U)}\le
 c_q'\|\tilde{v}\|_{L^2(J_n;U)},
 \label{eq:9f2}
\end{gather}
\end{subequations}
where $c'_q>0$ denotes a constant depending only on $q$. Function
 $\hat{v}$ is a modification of the discrete characteristic function of
 $\tilde{v}$ described in (2.8) presented in \cite{cw06}. The proof of
 \eqref{eq:9f2} is exactly the same as that of Lemma 2.4 of \cite{cw06}.   
Using \eqref{eq:71a}, \eqref{eq:9a}, \eqref{eq:9f0}, \eqref{eq:9f1} and
 applying the integration by parts, we can calculate as   
\begin{align*}
   \frac{1}{2}\|\tilde{v}(s)\|_{U}^2
  = \int_0^s(\tilde{v},\tilde{v}')_U~dt &= \int_{J_n}(\hat{v},\tilde{v}')_U~dt \\
  & =
 (\hat{v}(t_{n+1}),\tilde{v}(t_{n+1}))_U-\int_{J_n}(\hat{v}',\tilde{v})_U~dt \\
  & =
 \|\tilde{v}(t_{n+1})\|_U^2-\int_{J_n}(\hat{v}',v)_U~dt.
\end{align*}
 In the same way as the derivation of \eqref{eq:9i1}, we deduce by \eqref{eq:9f2}
\[
 \left|\int_{J_n}(\hat{v}',v)_U~dt\right|
 \le \frac{c_q}{2\tau_n}\|v\|_{L^2(J_n;U)}\|\hat{v}\|_{L^2(J_n;U)}
\le \frac{c_qc_q'}{2\tau_n}\|v\|_{L^2(J_n;U)}\|\tilde{v}\|_{L^2(J_n;U)}.
\]
Summing up those results, we obtain
 \[
  \|\tilde{v}(s)\|_U^2\le \frac{C_q}{\tau_n}\|v\|_{L^2(J_n;U)}\|\tilde{v}\|_{L^2(J_n;U)}
 \]
for $s\in J_n$, where $C_q>0$ denotes a constant depending only on $q$.  
 Integrating the both sides in $s\in J_n$, we deduce the desired
 inequality \eqref{eq:9b}.  
\end{proof}

We consider the trace inequality \eqref{eq:ch} for $T=1$ and write
$C_{\mathrm{Tr}}=C_{\mathrm{Tr},1}$, which is an absolute constant.
The scaling argument gives the following lemma.

\begin{lemma}
 \label{la:9}
 For $v\in L^2(J_n;V)\cap H^1(J_n;V')$, we have 
  \begin{equation}
   \max_{t_{n}\le t\le t_{n+1}}\|v(t)\|\le \frac{C_\mathrm{Tr}}{~\tau_n^{1/2}}\left(
\|v\|_{L^2(J_n;V)}^2+\tau_n^2\|v'\|_{L^2(J_n;V')}^2
					      \right)^{1/2}.
  \label{eq:13}
  \end{equation}
\end{lemma}

By virtue of \eqref{eq:00}, $A(t)$ is invertible for a.e. $t\in J$. 
Moreover, we have the following. 

\begin{lemma}
 \label{la:10}
\textup{(i)} $\displaystyle{\dual{g,A(t)^{-1}g}\ge \frac{\alpha}{M^2}\|g\|_{V'}}$ for all
 $g\in V'$ and a.e.~$t\in J$. \\
\textup{(ii)} $\displaystyle{\|A(t)^{-1}g\|_V\le \frac{1}{\alpha}\|g\|_{V'}}$
 for all $g\in V'$ and a.e.~$t\in J$.  
\end{lemma}

Now, we can state the following proof. 

\begin{proof}[Proof of Theorem \ref{th:0}, \eqref{eq:16}] 
Let $w_\tau\in \mathcal{S}_\tau$. 
First, we consider the case $q \ge 1$. Setting $\phi=A(t)^{-1}w_\tau'\in L^2(J;V)$, we define $\tilde{\phi}_\tau\in\mathcal{S}_\tau$ as
 $\tilde{\phi}_\tau|_{J_n}=\pi_n(\phi|_{J_n})$ for
 $n=0,\ldots,N-1$.

For abbreviation, we write $w=w_\tau$ and
 $\tilde{\phi}=\tilde{\phi}_\tau$. According to \eqref{eq:71a}, Lemmas \ref{la:6} and \ref{la:10}, we know
 \begin{subequations} 
  \label{eq:17}
  \begin{gather}
  \tilde{\phi}^{n,+}=0;\label{eq:17a}\\
 \int_{J_n}(\chi,\tilde{\phi})~dt=\int_{J_n}(\chi,\phi)~dt
  \quad (\forall \chi\in
  \mathcal{P}^{q-1}(J_n; V));\label{eq:17b}\\
  \|\tilde{\phi}\|_{L^2(J_n;V)}\le C_q\|A^{-1}w'\|_{L^2(J_n;V)}\le
  \frac{C_q}{\alpha}\|w'\|_{L^2(J_n;V')}. \label{eq:17c}
    \end{gather}
 \end{subequations} 

Now set $v=\tilde{\phi}+\mu w\in\mathcal{S}_\tau$ with $\mu>0$; the value of $\mu$ will be
 specified later.

 We prove
 \begin{equation}
  \label{eq:18}
   \normyb{v}\le C_1\normx{w},
 \end{equation}
 where $C_1$ is a positive constant depending only on $\mu,\alpha$ and $q$. Using \eqref{eq:17a}, \eqref{eq:17c} and Lemma \ref{la:9}, we can calculate
 \begin{align*}
  \normyb{v}^2
&\le \sum_{n=0}^{N-1}\int_{J_n}2\left[\|\tilde{\phi}\|_{V}^2+\mu^2\|w\|_V^2\right]~dt+\mu^2\sum_{n=1}^{N-1}\tau_n\|w^{n,+}\|^2+\mu^2\|w^{0,+}\|^2\\
&\le \sum_{n=0}^{N-1}\int_{J_n}\left[\left(2\frac{C_q^2}{\alpha^2}+\mu^2C_\mathrm{Tr}^2\tau_n^2\right)\|w'\|_{V'}^2+\mu^2\left(2+C_\mathrm{Tr}^2\right)\|w\|_V^2\right]~dt+\mu^2\|w^{0,+}\|^2,
 \end{align*}
which implies \eqref{eq:18}.

We apply \eqref{eq:17a}, \eqref{eq:17b}, \eqref{eq:17c}, Lemma \ref{la:10} and
 Young's inequality to obtain 
 \begin{align*}
  B_\tau(w,\tilde{\phi})
  &=  \sum_{n=0}^{N-1}\int_{J_n}\left[(w',A(t)^{-1}w')+
  \dual{A(t)w,\tilde{\phi}}\right]~dt\\ 
  &\ge  \sum_{n=0}^{N-1}\int_{J_n}\left[\frac{\alpha}{M^2}\|w'\|_{V'}^2
  -M\|w\|_V\|\tilde{\phi}\|_V\right]~dt \\
  &\ge  \sum_{n=0}^{N-1}\int_{J_n}\left[\frac{\alpha}{M^2}\|w'\|_{V'}^2
  -\frac{M^2}{2\delta}\|w\|_V^2-\frac{\delta}{2}\|\tilde{\phi}\|_V^2\right]~dt \\
  &\ge  \sum_{n=0}^{N-1}\int_{J_n}\left[\left(\frac{\alpha}{M^2}-\frac{C_q^2\delta}{\ns{2}\alpha^2}\right)\|w'\|_{V'}^2
  -\frac{M^2}{\ns{2}\delta}\|w\|_V^2\right]~dt ,
 \end{align*}
 where $\delta>0$ is arbitrary. 
 To estimate $B_\tau(w,\mu w)$, we recall the elementary identity
 \begin{equation}
\label{eq:id00} 
  (\chi^{n,+}-\chi^n,\chi^{n,+})=\frac{1}{2}\|\chi^{n,+}\|^2-\frac{1}{2}\|\chi^{n}\|^2+\frac{1}{2}\|\chi^{n,+}-\chi^n\|^2
\end{equation}
for $\chi\in\mathcal{S}_\tau$ and $n=0,\ldots,N-1$.
That is, we can calculate as
\begin{align}
 B_\tau(w,\mu w)
 &=\sum_{n=0}^{N-1}\int_{J_n}\left[ \frac{\mu}{2} \frac{d}{dt}\|w\|^2+
 \mu\dual{A(t)w,w} \right]~dt \nonumber \\
 &\mbox{ }\qquad +\frac{\mu}{2}\sum_{n=1}^{N-1}\left[\|w^{n,+}\|^2-\|w^{n}\|^2+\|w^{n,+}-w^n\|^2\right]+\mu\|w^{0,+}\|^2 \nonumber \\
 & \ge \sum_{n=0}^{N-1} \int_{J_n}\mu\alpha \|w\|_V^2 ~dt
 +\frac{\mu}{2}\sum_{n=1}^{N-1}\|w^{n,+}-w^n\|^2+\mu\|w^{0,+}\|^2. \label{eq:76}
\end{align}
Summing up, we deduce
 \begin{multline*}
  B_\tau(w,v)
  \ge
  \sum_{n=0}^{N-1}\int_{J_n}\left[\left(\frac{\alpha}{M^2}-\frac{C_q^2\delta}{\ns{2}\alpha^2}\right)\|w'\|_{V'}^2
  + 
  \left(\mu\alpha-\frac{M^2}{\ns{2}\delta}\right)\|w\|_V^2\right]~dt \\
  +\frac{\mu}{2}\sum_{n=1}^{N-1}\|w^{n,+}-w^n\|^2+\mu\|w^{0,+}\|^2.
 \end{multline*}
We take a suitably small $\delta$ and then choose a suitably large
 $\mu$. Consequently, there exists a positive constant $C_2$ depending
 only on $\alpha,M$ and $q$ such that
\[
  B_\tau(w,v)\ge C_2\normx{w}^2.
\] 
 This, together with \eqref{eq:18}, implies
\begin{equation}
\sup_{v\in\mathcal{S}_\tau}\frac{B_\tau(w,v)}{\normyb{v}} \ge C\normx{w} 
\label{eq:77}
\end{equation}
with $C=C_2/C_1$. \ns{Therefore, \eqref{eq:16} is proved.}

We proceed to the case $q=0$. Set $v=\mu w$ with $\mu>0$. Then, we deduce \eqref{eq:77}, because we still have \eqref{eq:18} and \eqref{eq:76}. 
\end{proof}

\begin{proof}[Proof of Theorem \ref{th:0}, \eqref{eq:23z}]
It suffices to prove that 
 \begin{subequations}
  \label{eq:23}
   \begin{gather}
  \exists c_2>0,\quad  \inf_{v_\tau\in\mathcal{S}_\tau}
       \sup_{w_\tau\in\mathcal{S}_\tau}\frac{B_\tau(w_\tau,v_\tau)}{\normyyb{w_\tau}\normxx{v_\tau}}=c_2;
    \label{eq:23a}\\
\ns{w_\tau\in \mathcal{S}_\tau,\quad (\forall v_\tau\in\mathcal{S}_\tau,\ B_\tau(w_\tau,v_\tau)=0)\quad
 \Longrightarrow \quad (w_\tau=0). }
  \label{eq:23b}    
   \end{gather}
 \end{subequations}
Actually, as recalled from the Introduction (i.e., the equivalence \eqref{eq:bnb} and
\eqref{eq:bnb3}), \eqref{eq:23} is equivalent to 
\[
 \inf_{w_\tau\in\mathcal{S}_\tau}
 \sup_{v_\tau\in\mathcal{S}_\tau}\frac{B_\tau(w_\tau,v_\tau)}{\normyyb{w_\tau}\normxx{v_\tau}}=
       \inf_{v_\tau\in\mathcal{S}_\tau}
       \sup_{w_\tau\in\mathcal{S}_\tau}\frac{B_\tau(w_\tau,v_\tau)}{\normyyb{w_\tau}\normxx{v_\tau}}=c_2.
\] 
Let $v_\tau\in \mathcal{S}_\tau$. 
Now, with a suitably large $\mu>0$, we set $w_\tau=-\tilde{\phi}_\tau+\mu v_\tau$ if $q\ge 1$ and $w_\tau=\mu v_\tau$ if $q=0$. Therein, $\tilde{\phi}_\tau\in\mathcal{S}_\tau$ is defined for $\phi_\tau=A^{-1}v_\tau'\in\mathcal{S}_\tau$ in the same way as in the proof of \eqref{eq:16}. 

 Then, in exactly the same way, we deduce
 \[
  B_\tau(w_\tau,v_\tau)\ge C\normxx{v_\tau}^2,\qquad \normyyb{w_\tau}\le C'\normxx{v_\tau},
 \]
and consequently obtain \eqref{eq:23a}. 

\ns{Finally, in order to prove \eqref{eq:23b}, we let $w_\tau\in \mathcal{S}_\tau$ satisfy 
$B_\tau(w_\tau,v_\tau)=0$ for $v_\tau\in\mathcal{S}_\tau$. Choosing $v_\tau=w_\tau$ and using \eqref{eq:0b} and \eqref{eq:id00}, we have 
\[
 \alpha \|w_\tau\|_{L^2(J;V)}^2+\frac{1}{2}\|w^N_\tau\|^2+\frac{1}{2}\|w^{0,+}_\tau\|^2+
\frac{1}{2}\sum_{n=1}^{N-1}\|w^{n,+}_\tau-w_\tau^n\|^2\le 0,
\]
which implies $w_\tau=0$. }
This completes the proof of \eqref{eq:23z}.
\end{proof}



\section{Application to the finite element method}
\label{sec:7}

This section presents application of our results, Theorems \ref{th:0}--\ref{th:2},
to error analysis of the finite element method. 
Letting $\Omega$ be a
polyhedral domain $\mathbb{R}^d$, $d=2,3$, with the boundary
$\partial\Omega$, we consider \ns{a second-order parabolic equation for the function $u=u(x,t)$ of $(x,t)\in\overline{\Omega}\times [0,T]$,
\begin{subequations}
\label{eq:h}
\begin{align}
\partial_t u &=  \partial_j (a_{ij}(x,t)\partial_iu)- a_i(x,t)\partial_i u-a_0(x,t)u+ f(x,t)  &&\textup{in } \Omega \times  J, \label{eq:h1}\\
u &= 0                             &&\textup{on }\partial\Omega\times J, \label{eq:h2}\\
u|_{t=0} &= u_0(x)                 &&\textup{on }\Omega,\label{eq:h3}
\end{align}
\end{subequations}
where $\partial_t u=\partial u/\partial t$ and $\partial_iu=\partial u/\partial x_i$ for $i=1,\ldots,d$. 
(The summation convention is employed throughout.) The coefficient functions $a_{ij}=a_{ji},a_{i}$ are assumed to be essentially bounded in $\Omega\times J$ and satisfy 
\begin{subequations}
\label{eq:ha}
\begin{align}
\exists \alpha >0,\quad a_{ij}(x,t)\xi_i\xi_j&\ge \alpha |\xi|^2_{\mathbb{R}^d}&& (x\in\Omega,~t\in J,~\xi=(\xi_i)\in\mathbb{R}^d);\label{eq:ha1}\\
a_0(x,t)-\frac12 \partial_i a_i(x,t)&\ge 0&& (x\in\Omega,~t\in J),\label{eq:ha2}
\end{align}
\end{subequations}
where $\alpha$ may depend on the final time $T$. Further assumptions will be mentioned later.  
}

We use the Lebesgue spaces $L^p=L^p(\Omega)$, $1\le p\le\infty$, and the
standard Sobolev spaces $H^k=H^{k}(\Omega)$, $k\ge 1$.
Let $H=L^2(\Omega)$ with the inner product $(v,w)=(v,w)_{L^2}$ and norm
$\|v\|=\|v\|_H=\|v\|_{L^2}$. Moreover, set $V=H^1_0=\{v\in H^1\mid v|_{\partial\Omega}=0\}$ with the norm
$\|v\|_V=\|v\|_{H^1_0}=\|\nabla v\|$.
The space $V'$ implies the dual
space $H^{-1}=H^{-1}(\Omega)$ of $H^1_0$ equipped with
the norm $\|v\|_{H^{-1}}=\sup_{w\in H^1_0}(v,w)/\|w\|_{H^1_0}$.  
\ns{We recall Poincar{\' e}'s inequality $\|v\|_{L^2}\le C_{\textup{P}}\|v\|_{H^1_0}$ for $v\in H^1_0$ with a positive constant $C_{\textup{P}}$.} 
The duality pairing between $H^{1}_0$ and $H^{-1}$ is denoted as
$\dual{w,v}={}_{H^{-1}}\dual{w,v}_{H^1_0}$. 
Spaces $\mathcal{X}=L^2(J;H^1_0)\cap
H^1(J;H^{-1})$, $\mathcal{Y}_1=L^2(J;H^1_0)$ and
$\mathcal{Y}=\mathcal{Y}_1 \times H$ are Hilbert spaces equipped with the norms
$\|v\|_{\mathcal{X}}^2=\|v\|_{L^2(J;H^1_0)}^2+\|\partial_t v\|_{L^2(J;H^{-1})}^2$,
$\|v\|_{\mathcal{Y}_1}^2=\|v\|_{L^2(J;H^1_0)}^2$ and
$\|v\|_{\mathcal{Y}}^2=\|v_1\|_{L^2(J;H^1_0)}^2+\|v_2\|^2$, 
respectively. 
 
For $t\in J$, the operator $A(t):H^1_0\to H^{-1}$ and
functional $F$ on $H^1_0$ are
defined as
\ns{
\begin{align*}
  \dual{A(t)w,v}&=a(t;w,v)=\int_\Omega [a_{ij}(x,t)(\partial_i w)(\partial_jv)+a_{i}(x,t)(\partial_i w)v+a_0(x,t)wv] ~dx&&
 (w,v\in H^1_0),\\
 \dual{F,v}&=(f,v)
 && (v\in H^1_0). 
\end{align*}
Using \eqref{eq:ha}, we have 
\begin{subequations} \
\label{eq:ha5}
\begin{align}
& a(t;w,v)\le M\|w\|_{H^1_0}\|v\|_{H^1_0}&&(t\in J,~v\in H_0^1),\label{eq:ha5a}\\
& a(t;v,v)
\ge \alpha \int_\Omega |\nabla v|^2~dx + \int_\Omega \left(a_0-\frac12 \partial_ia_i\right)v^2~dx \ge \alpha \|v\|_{H^1_0}^2&& (t\in J,~v\in H_0^1),\label{eq:ha5b}
\end{align}
\end{subequations}
where $M$ denotes a positive constant depending only on $\|a_{ij}\|_{L^\infty(J;L^\infty)}$, $\|a_i\|_{L^\infty(J;L^\infty)}$ and $C_{\textup{P}}$. 
}

With these interpretations, \eqref{eq:h} is converted into the abstract evolution equation of
\eqref{eq:0}. The weak formulation of \eqref{eq:h} is given as
follows: given
\begin{subequations} 
\label{eq:h1}
\begin{equation}
 f\in  L^2(J;L^2),\qquad u_0\in L^2,
 \label{eq:data}
\end{equation}
find $u\in \mathcal{X}$ such that
\begin{equation}
B(u,v)=\int_J(f,v_1)~dt
 + (u_0,v_2)\qquad (\forall v=(v_1,v_2)\in\mathcal{Y}),
\label{eq:h1a}
\end{equation}
where
\begin{equation}
 B(u,v)=
 \int_J \left[
	 \dual{\partial_tu,v_1}+\ns{a(t;u,v_1)}\right]~dt+(u(0),v_2)
\label{eq:h1b}
\end{equation}
\end{subequations}
for $u\in\mathcal{X}$, $v=(v_1,v_2)\in\mathcal{Y}$.

We proceed to the presentation of the finite element approximation. Let
$\{\mathcal{T}_h\}_h$ be a family of \emph{shape-regular} triangulation
of $\Omega$. The granularity parameter $h$ is defined as
${h=\max_{K\in\mathcal{T}_h}h_K}$, where $h_K$ denotes the
diameter of the circumscribed ball of $K$. For an integer $k\ge 1$, we
introduce the conforming $\mathrm{P}^k$ finite element space
\[
 X_h=X_h^k=\{v\in C^0(\overline{\Omega}) \mid
 v|_K\in\mathcal{P}^k(K;\mathbb{R})\ (\forall K\in\mathcal{T}_h),\
 v|_{\partial\Omega}=0\}\subset H_0^1.
\]

Recall that $\mathcal{S}_\tau$ is defined as \eqref{eq:S}. 
The space--time finite element space is given as
\[
 \mathcal{S}_{h,\tau}=
 \mathcal{S}_\tau^q(X_h,X_h).
\]
It is noteworthy that $\mathcal{S}_{h,\tau}\subset
\mathcal{S}_\tau$.

The dG$(q)$cG$(k)$ method reads: find
$u_{h,\tau}\in\mathcal{S}_{h,\tau}$ such that
\begin{subequations} %
\label{eq:h2}
\begin{equation}
B_{\tau}(u_{h,\tau},v)=
 \int_J(f,v)~dt
 + (u_{0},v^{0,+})\qquad (\forall v\in\mathcal{S}_{h,\tau}),
\label{eq:h2a}
 \end{equation}
where
\begin{equation}
B_{\tau}(u_{h,\tau},v)=  \sum_{n=0}^{N-1}\int_{J_n} \left[
	 (\partial_t u_{h,\tau},v)+\ns{a(t;u_{h,\tau},v)}\right]~dt+
(u^{0,+}_{h,\tau},v^{0,+})
 +\sum_{n=1}^{N-1}(u^{n,+}_{h,\tau}-u^n_{h,\tau},v^{n,+}) .
 \label{eq:h2b} 
\end{equation}
\end{subequations}

\begin{remark}
\label{rem:h1}To avoid unimportant difficulties, we set in \eqref{eq:h2} the same
 initial function $u_0$ as \eqref{eq:h1}. This fact implies that the initial
 value $u_{h,\tau}(0)\in X_h$ of the solution $u_{h,\tau}$ of \eqref{eq:h2} must be
 \[
  u_{h,\tau}(0)=P_hu_0,
 \]
 where $P_h$ denotes the $L^2$ projection of $L^{2}\to X_h$ defined as
 \begin{equation}
  \label{eq:l2p}
   (P_hw-w,v_h)=0\qquad (v_h\in X_h).
 \end{equation}
\end{remark}

In this setting, we have the consistency $B_{\tau}(u-u_{h,\tau},v)=0$ for all
$v\in \mathcal{S}_{h,\tau}$, where $u$ and $u_{h,\tau}$ respectively represent the
solutions of \eqref{eq:h1} and \eqref{eq:h2}. 
Therefore, in exactly the same way as for the proof of Theorems \ref{th:0}
and \ref{th:1}, 
we obtain the nearly best approximation inequalities such as \eqref{eq:100a} and \eqref{eq:100b}. 
Unfortunately, those results are useless for deducing explicit convergence rates directly. \ns{In fact, we used equalities $(I_nu)^{n,+}-u_{\tau}^{n,+}=0$ $(n=0,\ldots,N-1)$ in the proof of Theorem \ref{th:2}. But, we now do not know whether $(I_nu)^{n,+}-u_{h,\tau}^{n,+}$  $(n=0,\ldots,N-1)$ all vanish or not.}

 \ns{Consequently, we come to consider} a space semi-discrete scheme \eqref{eq:110} below. Set
$\mathcal{X}_h=H^1(J;X_h)$, $\mathcal{Y}_{1h}=L^2(J;X_h)$
 and $\mathcal{Y}_h=\mathcal{Y}_{1h}\times X_h$. They equip the norms 
$\|v_h\|_{\mathcal{X}_h}^2=\|\partial_tv_h\|_{L^2(J;X_h')}^2+\|v_h\|_{L^2(J;H^1_0)}^2$, 
 $\|v_{1h}\|_{\mathcal{Y}_{1h}}^2=\|v_{1h}\|_{\mathcal{Y}_1}^2$ and 
 $\|v_{2h}\|_{\mathcal{Y}_h}^2=\|v_{2h}\|_{\mathcal{Y}}^2$. Here and
 hereinafter, we write $\|v\|_{X_h'}=\sup_{\phi_h\in
 X_h}(v,\phi_h)/\|\phi_h\|_{H^1_0}$ for $v\in \ns{L^2}$. It is apparent that
\begin{equation}
 \label{eq:79}
  \|v\|_{X_h'}\le \|v\|_{H^{-1}}\quad (v\in \ns{L^2})
  ,\qquad\|v\|_{\mathcal{X}_h}\le \|v\|_{\mathcal{X}}\quad (v\in \mathcal{X}).
\end{equation}
It is noteworthy that $\mathcal{X}_h\subset \mathcal{X}$,
$\mathcal{Y}_{1h}\subset\mathcal{Y}_1$ and 
$\mathcal{Y}_h\subset\mathcal{Y}$. 
We consider the problem to find $u_h\in\mathcal{X}_h$ such that 
\begin{subequations} 
\label{eq:110}
\begin{equation}
B(u_h,v_h)=\int_J(f,v_{1h})~dt
 + (u_0,v_{2h})\qquad (\forall v_h=(v_{1h},v_{2h})\in\mathcal{Y}_h),
\label{eq:110a}
\end{equation}
where
\begin{equation}
 B(u_h,v_h)=
 \int_J \left[
	 (\partial_tu_h,v_{1h})+\ns{a(t;u_h,v_{1h})}\right]~dt+(u_h(0),v_{2h}).
\label{eq:110b}
\end{equation}
\end{subequations}
Therein, $u_h(0)\in L^2$ in \eqref{eq:110b} is well-defined, because there exists a positive constant $C_{\mathrm{Tr},T}'$ depending only on $T$ such that 
  \begin{equation}
\max_{t\in \overline{J}}\|v_h(t)\| \le C_{\mathrm{Tr},T}'\|v_h\|_{\mathcal{X}_h}\qquad
 (v_h\in \mathcal{X}_h).
  \label{eq:ch1}
  \end{equation}
\ns{
This inequality does not follow directly from \eqref{eq:ch}. 
However, the proof is exactly the same as that of \eqref{eq:ch}.  
}

Introducing\ns{, for $t\in J$,} 
the discrete \ns{elliptic operator} $A_h(t):X_h\to X_h$ by
\[
\ns{ (A_h(t)w_{h},v_h)=a(t;w_h, v_h)}\qquad (w_h,v_h\in X_h),
\]
the problem \eqref{eq:110} is expressed equivalently as
\begin{equation}
 \label{eq:110a}
\frac{d}{dt}u_{h}(t)+\ns{A_h(t)}u_h(t)=P_hf(t)\quad (t\in J);\quad u_h(0)=P_hu_0,
\end{equation}
where $P_h$ denotes the $L^2$ projection defined as \eqref{eq:l2p}.

Because \eqref{eq:h2} is regarded as a time discretization scheme to
\eqref{eq:110}, we can apply Theorem \ref{th:2} directly to obtain the
following lemma. 
Here and hereinafter, we use $C$ to represent general constants independent of $h$ and $\tau$. 

\begin{lemma}
 \label{la:21}
Let $u_{h,\tau}$ and $u_h$ respectively represent the solutions of \eqref{eq:h2} and
 \eqref{eq:110}. If $q\ge 0$ and $\partial_t^{q}u_h\in \mathcal{X}_h$, then 
\begin{equation*}
\left(\sum_{n=0}^{N-1}\|\partial_tu_h-\partial_tu_{h,\tau}\|_{L^2(J_n;X_h')}^{2}\right)^{1/2}
\ns{+\|u_h-u_{h,\tau}\|_{L^2(J;H^1_0)}}
 \le C\tau^{q}
\|\partial_t^{q}u_h\|_{\mathcal{X}_h}.
\end{equation*}
Moreover, if $q\ge 1$ and $\partial_t^{q+1}u_h\in \mathcal{Y}_{1h}$, then  
 \begin{equation}
  \sup_{1\le n\le N}\|u_h(t_n)-u_{h,\tau}(t_n)\|+\|u_h-u_{h,\tau}\|_{L^2(J;H^1_0)}
  \le  C\tau^{q+1}\|\partial_t^{q+1}u_h\|_{\mathcal{Y}_{1h}}.
  \label{eq:121b}
  \end{equation}

If $q=0$ and $\partial_t u_h\in \mathcal{X}_h$, the
 estimate \eqref{eq:121b} holds with replacement
 $\|\partial_t^{q+1}u_h\|_{\mathcal{Y}_{1h}}$ by
 $\|\partial_t u_h\|_{\mathcal{X}_h}$. 
\end{lemma}

\ns{
Below, we study the stability of $u_h$ and error $u-u_h$ in
various norms. First, in order to derive the stability estimates, we need the following lemma.  

  \begin{lemma}
   There exists a positive constant $\gamma$ which is independent of $h$
   such that  
\begin{equation}
\inf_{w_h\in\mathcal{X}_h}\sup_{v_h\in\mathcal{Y}_h}\frac{B(w_h,v_h)}{\|w_h\|_{\mathcal{X}_h}\|v_h\|_{\mathcal{Y}_h}}= \gamma .
    \label{eq:infsup1}
\end{equation}
  \label{la:infsup1}
  \end{lemma}

This result is not new by itself; it is mentioned, for example, in \cite{and13} and \cite{tv16} as for a sightly different setting and assumptions.  
We will present a version of the proof in Appendix \ref{sec:a} for completeness.  
As an application of this lemma, we prove the following result. 
 
     \begin{lemma}
      Letting $q$ be a non-negative integer, we assume that
     \begin{equation}
      \label{eq:ck}
       \partial_t^la_{ij},\
       \partial_t^la_{i},\
       \partial_t^la_{0} \in L^\infty(\Omega\times J)\quad (1\le i,j\le d,~0\le l\le q).
     \end{equation}
 Let $u$ and $u_h$ respectively represent the solutions of \eqref{eq:h1}
    and \eqref{eq:110}. 
    If $u\in H^q(J;H^1_0)$ and $\partial_t^{q+1} u\in L^2(J;H^{-1})$, then we have 
    \begin{equation}
     \|\partial_t^{q+1}u_h\|_{L^2(J;X_h')}^2 + 
\|u_h\|_{H^{q}(J;H^1_0)}^2
      \le           C\left(\|\partial_t^{q+1}u\|_{L^2(J;H^{-1})}^2+
 \|u\|_{H^{q}(J;H^1_0)}^2\right)     .
\label{eq:131b}
\end{equation}
\label{la:23}
     \end{lemma}

\begin{proof}
Because $\mathcal{X}_h\subset \mathcal{X}$ and
    $\mathcal{Y}_h\subset\mathcal{Y}$, we have the consistency
    \begin{equation}
     B(u-u_h,v_h)=0\qquad (v_h\in\mathcal{Y}_h).
     \label{eq:141a}
    \end{equation}
Recall that $u_h(0)$ is chosen as $(u_h(0),v_h)=(u_0,v_h)$ for any $v_h\in X_h$. 
First, we consider the case $q=0$. Using \eqref{eq:infsup1} and \eqref{eq:141a}, we can perform the estimations as 
\[
  \|u_h\|_{\mathcal{X}_h} 
\le \frac{1}{\gamma}\sup_{v_h\in\mathcal{Y}_h}\frac{B(u_h,v_h)}{\|v_h\|_{\mathcal{Y}_h}} 
= \frac{1}{\gamma}\sup_{v_h\in\mathcal{Y}_h}\frac{B(u,v_h)}{\|v_h\|_{\mathcal{Y}_h}} 
\le \frac{1}{\gamma}\|u\|_{\mathcal{X}_h}
\le \frac{1}{\gamma}\|u\|_{\mathcal{X}},
\]
which implies \eqref{eq:131b} for $q=0$. 

We proceed to the case $q\ge 1$. Assuming that \eqref{eq:ck} is sastisfied, we set 
\[
 \Phi_q(w,v)=\sum_{l=0}^{q-1}  \binom{q}{l} \int_Ja^{(q-l)}(t;\partial_t^{l}w,v)~dt
\qquad (w\in H^{q-1}(J;H_0^1),~v\in L^2(J;H_0^1)),
\]
where $\binom{q}{l}={q!}/(l!(q-l)!)$ and 
\[
 a^{(l)}(t;w,v)=\int_\Omega 
\left[
(\partial_t^la_{ij})(\partial_iw)(\partial_jv)
+(\partial_t^la_{i})(\partial_iw)v
+(\partial_t^la_{0})wv
\right]~dt.
\]
It is apparent that 
\begin{equation*}
|\Phi_q(w,v)|\le C \|w\|_{H^{q-1}(J;H_0^1)}\|v\|_{L^2(J;H_0^1)}.
\end{equation*}

Now assume that $u\in H^q(J;H^1_0)$ and $\partial_t^{q+1}u\in L^2(J;H^{-1})$. 
Let 
$v_{1h}\in C_0^\infty(J;X_h)$ and 
$v_{2h}\in X_h$.  
Substituting $v_h=((-1)^q\partial_t^{q}v_{1h},v_{2h})$ for \eqref{eq:141a}, we have by the integration by parts  
    \begin{equation*}
 {B}(\partial_t^{q}u-\partial_t^{q}u_h,v_{h})+\Phi_q(u-u_h,v_{1h})=0.
    \end{equation*}
By the density, the identity holds for any $v_{h}\in
\mathcal{Y}_{h}$. (It is noteworthy that $C_0^\infty(J;X_h)$ is dense in $\mathcal{Y}_{1h}$.)
We apply \eqref{eq:infsup1} again to obtain 
\begin{align*}
  \|\partial_t^{q}u_h\|_{\mathcal{X}_{h}}&\le
 \frac{1}{\gamma}\sup_{v_{h}\in\mathcal{Y}_{h}}\frac{B(\partial_t^{q}u_h,v_{h})-\Phi_q(u-u_h,v_{1h})}{\|v_{h}\|_{\mathcal{Y}_h}}\\
& \le
C
\left( \|\partial_t^{q}u\|_{\mathcal{X}}+\|u-u_h\|_{H^{q-1}(J;H_0^1)}\right)
\end{align*}
and, consequently, 
\begin{multline*}
 \|\partial_t^{q+1}u_h\|_{L^2(J;X_h')}^2+
\|\partial_t^{q}u_h\|_{L^2(J;H_0^1)}^2 \\ 
\le  C\left(\|\partial_t^{q+1}u\|_{L^2(J;H^{-1})}^2+
\|\partial_t^{q}u\|_{L^2(J;H_0^1)}^2
+\|u\|_{H^{q-1}(J;H_0^1)}^2+\|u_h\|_{H^{q-1}(J;H_0^1)}^2
\right).
\end{multline*}
Therefore, we deduce \eqref{eq:131b} by the induction in $q$.  
 \end{proof}

}

 \ns{
Next, we study error estimates for the space semi-discrete scheme \eqref{eq:110a}. 
This subject has a long history. Non-smooth data error estimates have been examined by many researchers; see 
the monographs, \cite{fss01}, \cite{tho06}, for more detail. Recently, some results based on the inf-sup conditions are reported; see \cite{and13} and \cite{tv16} for example. In particular, \cite{tv16} derived the error estimate \eqref{eq:120a} below under slightly different assumptions.  The regularity assumption on $u$ is weaker than ours. On the other hand, they assumed that the $L^2$ projection $P_h$ is $H^1_0$ stable, that is, there exists a constant $c_*>0$ which is independent of $h$ such that 
\begin{equation}
 \label{eq:stab}
  \|P_hv_h\|_{H^1_0}\le c_*\|v_h\|_{H^1_0}\qquad (v_h\in X_h).  
\end{equation}
It is noteworthy that this inequality actually holds if
  $\{\mathcal{T}_h\}_h$ satisfies some shape conditions. 
Inequality \eqref{eq:stab} brings us many useful results. One of that is 
\begin{equation}
 \label{eq:66}
\|v_h\|_{H^{-1}}\le
  c_*^{-1}\|v_h\|_{X_h'}\qquad (v_h\in X_h). 
\end{equation}
That is, $\|\cdot\|_{H^{-1}}$ and $\|\cdot\|_{X_h'}$ are equivalent norms on $X_h$. 
\cite{bj89} assumed \eqref{eq:66} in their analysis of the $hp$ finite element method. 
\cite{and13} considered a version of \eqref{eq:66} in his study on the space-time finite element method.

Below we do not assume \eqref{eq:stab} and derive error estimates for $u_h$. Then, we will mention how our results being improved under \eqref{eq:stab}; see Remark \ref{rem:7}. 

Applying the integration by parts formula (see Theorem 2, \S XVIII-1, of \cite{dl92} for example), 
we see that $u-u_h$ satisfies 
\begin{equation}
\label{eq:a46}
 \int_J [\dual{u-u_h,-\partial_tv_h}+a(t;u-u_h,v_h)]~dt+(u(T)-u_h(T),v_h(T))=0
\qquad (v_h\in\mathcal{X}_h).
\end{equation} 
Motivated by this observation, we set 
\[
      \tilde{B}(w,v)=\int_J [\dual{w_1,-\partial_tv}+a(t;w_1,v)]~dt+(w_2,v(T))
\quad (w=(w_1,w_2)\in\mathcal{Y},v\in\mathcal{X}).
\]
The identity \eqref{eq:a46} implies 
\begin{equation}
\tilde{B}(e_h,v_h)=0\qquad (v_h\in\mathcal{X}_h),
 \label{eq:141h}
\end{equation}
where $e_h=(u-u_h,u(T)-u_h(T))\in\mathcal{X}$. 

The proof of the following lemma is also postponed for Appendix \ref{sec:a}. 
 
\begin{lemma}
   There exists a positive constant $\tilde{\gamma}$ which is independent of $h$
   such that  
\begin{equation}
\inf_{w_h\in\mathcal{Y}_{h}} 
\sup_{v_h\in\mathcal{X}_h}
\frac{\tilde{B}(w_h,v_h)}{\|w_h\|_{\mathcal{Y}_{h}}\|v_h\|_{\mathcal{X}_h}}= \tilde{\gamma}.
    \label{eq:infsup2}
\end{equation}
  \label{la:infsup2}
  \end{lemma}

Using \eqref{eq:79}, \eqref{eq:infsup1}, \eqref{eq:141a}, \eqref{eq:141h} and \eqref{eq:infsup2}, we get
    \begin{align*}
 \|u-u_h\|_{\mathcal{X}_h} & \le C\inf_{w_h\in \mathcal{X}_h}\|u-w_h\|_{\mathcal{X}}, \\
 \|(u-u_h,u(T)-u_h(T))\|_{\mathcal{Y}_h} 
& \le C\inf_{(v_{1h},v_{2h})\in \mathcal{Y}_h}\|(u-v_{1h},u(T)-v_{2h})\|_{\mathcal{Y}} 
    \end{align*}
in the similar way as that used for the proof of Theorem \ref{th:1}. 
On the other hand, we know  
\[
 \inf_{v_h\in X_h}\|w-v_h\|_{L^2} \le Ch^{k}|w|_{H^k},\qquad 
 \inf_{v_h\in X_h}\|w-v_h\|_{H^1_0} \le Ch^{k}|w|_{H^{k+1}}
\]
in view of the standard error estimations for the Lagrange and Scott--Zhang interpolations.   
Combining these results, we obtain the following optimal-order error estimates for the space discretization. For a positive integer $k$, we write  
\[
|v|_{L^2(J;H^{k})}^2=
\int_J\Big(\sum_{|\alpha|=k}\|\partial^\alpha v\|_{L^2}^2\Big)^{1/2}~dt,\quad 
|v|_{L^\infty(J;H^{k})}=\sup_{t\in J}
\Big(\sum_{|\alpha|=k}\|\partial^\alpha v\|_{L^2}^2\Big)^{1/2},
\]
where $\alpha=(\alpha_1,\ldots,\alpha_{d})$ denotes the multi-index. 

  \begin{lemma}
 Let $u$ and $u_h$ respectively represent the solutions of \eqref{eq:h1}
   and \eqref{eq:110}. If $u\in H^1(J;H^{k})\cap L^2(J;H^{k+1})$, we
   have 
   \begin{subequations} %
\label{eq:122}
\begin{equation}
\|\partial_tu-\partial_tu_{h}\|_{L^2(J;X_h')}+\|u-u_{h}\|_{L^2(J;H^1_0)}
 \le Ch^{k} \left(|\partial_tu|_{L^2(J;H^{k})}+|u|_{L^2(J;H^{k+1})}\right).
 \label{eq:122a} 
\end{equation}
Moreover, if $u\in L^2(J;H^{k+1})\cap L^\infty(J;H^{k})$, we have
\begin{equation}
 \|u-u_h\|_{L^\infty(J;L^2)}+\|u-u_{h}\|_{L^2(J;H^1_0)} \le Ch^k
 \left( |u|_{L^\infty(J;H^{k})}+|u|_{L^2(J;H^{k+1})}\right).
 \label{eq:122b}
\end{equation}
\end{subequations}
\label{la:22}
\end{lemma}


Summing up those lemmas, we obtain the following theorem as the final
result of this paper. Let
\begin{align*}
  \|v\|_{Z_1}&=\|\partial_t^{q+1}v\|_{L^2(J;H^{-1})}+\|v\|_{H^q(J;H^{1}_0)}+ |\partial_t v|_{L^2(J;H^{k})}+ |v|_{L^2(J;H^{k+1})},\\
 \|v\|_{Z_2}&=
 \|\partial_t^{q+2}v\|_{L^2(J;H^{-1})}+\|v\|_{H^{q+1}(J;H^{1}_0)}+ |v|_{L^\infty(J;H^{k})}+ |v|_{L^2(J;H^{k+1})}  
\end{align*}
for a sufficiently regular function $v$.

\begin{theorem}
 \label{th:20}
 Letting $k\ge 1$ and $q\ge 0$ be integers, we suppose that $u\in\mathcal{X}$ and $u_{h,\tau}\in\mathcal{S}_{h,\tau}$ are the
 respective solutions of \eqref{eq:h1} and \eqref{eq:h2}. 
Assume, moreover, that the coefficient functions satisfy the condition \eqref{eq:ck} and that $u$ is sufficiently regular that $\|u\|_{Z_1},\|u\|_{Z_2}<\infty$. Then, there exist positive constants $c_5,c_6$ which are independent of $h$ and $\tau$ such that 
\begin{subequations} 
 \label{eq:120}
\begin{align}
\left(\sum_{n=0}^{N-1}\|\partial_tu-\partial_tu_{h,\tau}\|_{L^2(J_n;X_h')}^2\right)^{1/2}
 +\|u-u_{h,\tau}\|_{L^2(J;H^1_0)} &\le c_5(h^k+\tau^{q})\|u\|_{{Z}_1}, \label{eq:120a}\\
\sup_{1\le n\le N}\|u(t_n)-u_{h,\tau}(t_n)\|+ \|u-u_{h,\tau}\|_{L^2(J;H^1_0)}&\le c_6(h^k+\tau^{q+1})\|u\|_{{Z}_2}.\label{eq:120b}
\end{align}
\end{subequations}
\end{theorem}

}

 \begin{remark}
  \label{rem:7}
\ns{We assume $u\in H^1(J;H^{k})\cap L^2(J;H^{k+1})$ in deriving the error estimate \eqref{eq:122a}, while \cite{tv16} assume only $u\in H^1(J;H^{k-1})\cap L^2(J;H^{k+1})$.
} 
If $P_h$ is stable in the $H^1_0$ norm, that is, \eqref{eq:stab} is satisfied, 
the right-hand side of \eqref{eq:122a} is replaced by 
\[
  Ch^{k} \left(|\partial_tu|_{L^2(J;H^{k-1})}+|u|_{L^2(J;H^{k+1})}\right).
\]
In fact, if this is the case, we can apply $\|w-P_hw\|_{H^{-1}}\le Ch^{k}|w|_{H^{k-1}}$ and $\|w-P_hw\|_{H^{1}_0}\le Ch^{k}|w|_{H^{k+1}}$; 
see (2.2) of \cite{ch02} or Theorem 4.1 of \cite{ba72}. 
Moreover, the first estimate \eqref{eq:120a} in Theorem \ref{th:20} is improved as
\begin{equation*}
 \left(\sum_{n=0}^{N-1}\|\partial_tu-\partial_tu_{h,\tau}\|_{L^2(J_n;H^{-1})}^2\right)^{1/2}\ns{+\|u-u_{h,\tau}\|_{L^2(J;H^1_0)}}
  \le c_5'(h^k+\tau^{q})\|u\|_{{Z}_3},
\end{equation*}
where $\|\cdot\|_{Z_3}$ is an obvious modification of $\|\cdot
  \|_{{Z_1}}$. In fact, dividing $u-u_{h,\tau}$ into
  $(u-P_hu)+(P_hu-u_{h,\tau})$ and applying \eqref{eq:66}, we deduce   
  \begin{multline*}
   \left(\sum_{n=0}^{N-1}\|\partial_tu-\partial_tu_{h,\tau}\|_{L^2(J_n;H^{-1})}^2\right)^{1/2}
   \\
   \le
   \left(\sum_{n=0}^{N-1}\|\partial_tu-P_h\partial_tu\|_{L^2(J_n;H^{-1})}^2\right)^{1/2}+
   \frac{1}{c_*}\left(\sum_{n=0}^{N-1}\|\partial_tu-\partial_tu_{h,\tau}\|_{L^2(J_n;X_h')}^2\right)^{1/2}.
  \end{multline*}
 \end{remark}


\paragraph{Acknowledgements.}

\ns{I thank the anonymous reviewers for their valuable comments and
suggestions to improve the quality of the paper.} This work was supported by JST CREST Grant Number JPMJCR15D1, Japan and by JSPS KAKENHI Grant Numbers 15H03635, Japan.


\appendix
\section{Proofs of Lemmas \ref{la:infsup1} and \ref{la:infsup2}}
\label{sec:a}

\ns{
  \begin{proof}[Proof of Lemma \ref{la:infsup1}]
We use an
   equivalent norm to $\|\cdot\|_{\mathcal{X}_h}$ given as  
   \[
   \|v\|_{\mathcal{X}_h,\diamondsuit}^2=
   \int_J \|\partial_tv+A_h(t)v\|_{X_h'}^2~dt+\|v(0)\|^2\qquad (v\in \mathcal{X}_h).
   \]
In fact, we have 
\[
   \alpha \|v\|_{\mathcal{X}_h}\le \|v\|_{\mathcal{X}_h,\diamondsuit}\le
   M'\|v\|_{\mathcal{X}_h}\quad (v\in \mathcal{X}_h),
\]
where $M'$ denotes a positive constant depending only on $M$ and
   ${C}_{\textup{Tr}}'$. 

   Moreover, a discrete version of Lemma \ref{la:10} is available. That
   is, 
   \begin{equation}
    \label{eq:73}
\|A_h(t)^{-1}g\|_{H^1_0}\le \frac{1}{\alpha}\|g\|_{X_h'},\quad 
 (g,A_h(t)^{-1}g)\ge \frac{\alpha}{M^2}\|g\|_{X_h'}^2 \qquad 
      (g\in X_h,~t\in J).
\end{equation}
  
Now let $w\in\mathcal{X}_h$ be arbitrary and set
   $v=(v_{1},v_{2})\in\mathcal{Y}_h$, 
   $v_{1}=A_h(t)^{-1}\partial_tw+w$, $v_{2}=w(0)$. 
   Using \eqref{eq:73}, we can perform the estimations
 \[
    \|v\|_{\mathcal{Y}_h}^2
  =\int_J \|A_h(t)^{-1}[\partial_tw+A_h(t)w]\|_{X_h'}^2~dt+\|w(0)\|^2
  \le \frac{1}{\alpha^2}\|w\|_{\mathcal{X}_h,\diamondsuit}^2,
 \]
   and
   \begin{align*}
    B(w,v)&\ge \frac{\alpha}{M^2}\|w\|_{\mathcal{X}_h,\diamondsuit}^2
    \ge  \frac{\alpha^2}{M^2}\|w\|_{\mathcal{X}_h,\diamondsuit}\|v\|_{\mathcal{Y}_h}.
   \end{align*}
Combining these estimates, we deduce \eqref{eq:infsup1}. 
   \end{proof}
}

\medskip

\ns{
\begin{proof}[Proof of Lemma \ref{la:infsup2}]
The direct proof of \eqref{eq:infsup2} is apparently so difficult
that we take a detour. We will show: 
\begin{subequations} %
\begin{gather}
\exists  \tilde{\beta}>0,\quad
 \inf_{v\in\mathcal{X}_h}\sup_{w\in\mathcal{Y}_{h}}\frac{\tilde{B}(w,v)}{\|w\|_{\mathcal{Y}_{h}}\|v\|_{\mathcal{X}_h}}= \tilde{\beta};
 \label{eq:151i}\\
w\in \mathcal{Y}_h,\quad (\forall v\in\mathcal{X}_{h},\ \tilde{B}(w,v)=0)\quad
 \Longrightarrow \quad (w_h=0). 
 \label{eq:151j}
\end{gather}
\end{subequations}
Then, the general theory engenders \eqref{eq:infsup2}; recall the
equivalence \eqref{eq:bnb} and 
\eqref{eq:bnb3} described in the Introduction.   
   
We use another equivalent norm to $\|\cdot\|_{\mathcal{X}_h}$ given as  
   \[
   \|v\|_{\mathcal{X}_h,\blacklozenge}^2=
   \int_J \|-\partial_tv+A_h(t)^*v\|_{X_h'}^2~dt+\|v(T)\|^2\qquad (v_h\in \mathcal{X}_h),
   \]
where $A_h(t)^*$ denotes the adjoint operator of $A_h(t)$; $(A_h(t)w,v)=(w,A_h(t)^*v)$ for $w,v\in X_h$ and $t\in J$. We have 
$\alpha \|v\|_{\mathcal{X}_h}\le \|v\|_{\mathcal{X}_h,\blacklozenge}\le
   M'\|v\|_{\mathcal{X}_h}$ for $v\in \mathcal{X}_h$. Inequalities \eqref{eq:73} are valid for $A_h(t)^*$. 

The proof of \eqref{eq:151i} is exactly the same as that of \eqref{eq:infsup1};  
we choose $w=(w_{1},w_{2})\in \mathcal{Y}_h$ as $w_{1}=-[A_h(t)^*]^{-1}\partial_tv+v$ and $w_{2}=v(T)$ for a given $v\in\mathcal{X}_h$.   

To prove \eqref{eq:151j}, we assume that $w=(w_{1},w_{2})\in \mathcal{Y}_h$ solves 
\begin{equation}
 \label{eq:151m}
\int_J [(w_1,-\partial_tv)+a(t;w_{1},v)]~dt+(w_2,v_h(T))=0
\quad (v\in \mathcal{X}_h).
\end{equation}

First, considering $v\in\mathcal{X}_h$ such that $v=0$ on $(0,T-\ep)$ and $v(T)=w_2$ for a sufficiently small $\ep>0$, we obtain $w_2=0$. 

Substituting $v=\varphi\in C_0^\infty(J;X_h)$ for \eqref{eq:151m}, we see that 
$w_1\in H^1(J;X_h)$ and $\partial_tw_1=-A_h(t)w_1$. Moreover, because $L^2(J;X_h)$ is dense in $C_0^\infty(J;X_h)$, we have 
\[
 \int_J [(\partial_t w_1,\varphi)+(A_h(t)w_1,\varphi)]~dt=0
\quad (\varphi\in L^2(J;X_h)).
\]
At this stage, let $\eta\in X_h$ be arbitrary. Substituting $\varphi=t\eta\in L^2(J;X_h)$ and applying integration by parts, we deduce
\[
(w_1(T),T\eta)+\int_J [(w_1,-\partial_t (t\eta))+a(t;w_1,t\eta)]~dt=0. 
\]
Thefore, in view of \eqref{eq:151m} and $w_2=0$, we deduce $w_1(T)=0$. 
Now, substituting $v=w_1$ for \eqref{eq:151m}, 
\[
 \frac{1}{2}\|w_1(0)\|^2+\alpha \|w_1\|_{L^2(J;H_0^1)}^2\le 0,
\]
which gives $w_1=0$. This completes the proof.  
\end{proof}
}

\bibliographystyle{plain}


\end{document}